\documentclass{article}

\usepackage[all]{xy}
\xyoption{arc}
\xyoption{rotate}
\xyoption{2cell}
\UseAllTwocells

\usepackage{graphicx}
\usepackage{color}
\usepackage{amssymb,amsmath}
\usepackage{mathrsfs}
\usepackage{bbm}
\usepackage[text={140mm,220mm},centering]{geometry}

\usepackage{tocloft}
\usepackage[numbers]{natbib}
\usepackage{natbibspacing}
\renewcommand{\bibfont}{\small}

\definecolor{myurlcolor}{rgb}{0,0,0.5}
\usepackage{hyperref}
\hypersetup{colorlinks,
linkcolor=black,
citecolor=black,
urlcolor=myurlcolor}

\newcommand{\blank}{(\hspace*{1ex})}
\newcommand{\dashbk}{-}
\newcommand{\hyph}{\mbox{-}}

\newcommand{\ucontents}[2]{\addcontentsline{toc}{#1}{\numberline{}{#2}}}
\newcommand{\cat}[1]{\mathscr{#1}}
\newcommand{\fcat}[1]{\mathbf{#1}}
\newcommand{\scat}[1]{\mathbf{#1}}

\newcommand{\such}{:}
\newcommand{\without}{\setminus}

\newcommand{\Hom}{\mathrm{Hom}}
\newcommand{\op}{\mathrm{op}}
\newcommand{\id}{\mathrm{id}}
\newcommand{\Alg}{\fcat{Alg}}
\newcommand{\Set}{\fcat{Set}}
\newcommand{\Sheaf}{\fcat{Sheaf}}
\newcommand{\CHSheaf}{\fcat{CHSheaf}}
\newcommand{\integers}{\mathbb{Z}}

\newcommand{\ftrcat}[2]{[#1,#2]}

\newcommand{\CAT}{\fcat{CAT}}
\newcommand{\Mod}{\mathbf{Mod}}
\newcommand{\reals}{\mathbb{R}}
\newcommand{\rationals}{\mathbb{Q}}
\newcommand{\latin}[1]{#1}	
\newcommand{\swnt}{\rotatebox{45}{$\Leftarrow$}\!}

\newcommand{\demph}[1]{\textbf{\textup{#1}}}
\newcommand{\done}{\hfill\ensuremath{\Box}}
\newenvironment{prooflike}[1]{\begin{trivlist}\item\textbf{#1}\ }
{\end{trivlist}}
\newenvironment{proof}{\begin{prooflike}{Proof}}{\end{prooflike}}
\newcommand{\End}{\fcat{End}}
\newcommand{\iso}{\cong}
\newcommand{\nat}{\mathbb{N}}	
\newcommand{\eqv}{\simeq}
\newcommand{\pr}{\mathrm{pr}}
\newcommand{\copr}{\mathrm{copr}}
\newcounter{bean}
\newcommand{\of}{\,\raisebox{0.08ex}{\ensuremath{\scriptstyle\circ}}\,}
\newcommand{\sub}{\subseteq}
\newcommand{\One}{\mathbf{1}}
\renewcommand{\to}{\longrightarrow}
\newcommand{\toby}[1]{\stackrel{#1}{\to}}
\newcommand{\Open}{\fcat{O}}
\newcommand{\FinSet}{\fcat{FinSet}}
\newcommand{\Top}{\fcat{Top}}

\newcommand{\Vect}{\fcat{Vect}}
\newcommand{\LCVect}{\fcat{LCVect}}
\newcommand{\dee}{\,d}
\newcommand{\incl}{\hookrightarrow}
\newcommand{\epic}{\twoheadrightarrow}
\renewcommand{\implies}{\Rightarrow}
\newcommand{\uf}[1]{\mathscr{#1}}
\DeclareMathOperator{\Simp}{Simp}
\DeclareMathOperator{\im}{im}
\newenvironment{ufchar}{\begin{center}\it}{\end{center}}
\newcommand{\from}{\colon}
\newcommand{\mnd}[1]{\mathbf{#1}}
\newcommand{\Mnd}{\fcat{Mnd}}
\newcommand{\Grp}{\fcat{Grp}}
\newcommand{\Ring}{\fcat{Ring}}
\newcommand{\FDVect}{\fcat{FDVect}}
\newcommand{\Field}{\fcat{Field}}
\newcommand{\CptHff}{\fcat{CptHff}}
\newcommand{\undvect}[1]{|#1|}
\newcommand{\fmc}[1]{\cat{K}(#1)}
\DeclareMathOperator{\Frac}{Frac}
\DeclareMathOperator{\Spec}{Spec}
\DeclareMathOperator{\coker}{coker}
\newcommand{\Fam}{\fcat{Fam}}
\newcommand{\FinFam}{\fcat{FinFam}}
\newcommand{\cm}[1]{(\CAT/{#1})_\textup{CM}}
\newcommand{\monadic}[1]{(\CAT/{#1})_\textup{mndc}}
\newcommand{\bl}{{\scriptscriptstyle\bullet}}
\DeclareMathOperator{\Eq}{Eq}
\newcommand{\ideal}[1]{\mathfrak{#1}}
\newcommand{\idashbk}{\!-}
\newcommand{\textiff}{\Leftrightarrow}
\newcommand{\ulp}[1]{\prod\nolimits_{#1}} 
\newcommand{\aulp}[1]{\bigwedge\nolimits_{#1}} 
\newcommand{\END}{\mathbf{\underline{End}}}

\newcommand{\lt}[1]{\lim_{\substack{\longleftarrow\\ #1}}}
\newcommand{\colt}[1]{\lim_{\substack{\longrightarrow\\ #1}}}
\newcommand{\elt}[1]{\scat{E}(#1)}
\newenvironment{centerdiag}{%
\begin{array}{c}%
\hspace*{-.5em}}%
{\hspace*{-.5em}%
\end{array}}
\newenvironment{mapdefn}{%
\begin{array}{ccc}}%
{\end{array}%
}

\newcommand{\slob}[3]{%
\hspace{-.5em}%
\begin{array}{c}
\xymatrix{#1 \ar[d]^{#3}\\ #2}
\end{array}%
\hspace{-.5em}}

\newtheorem{thm}{Theorem}[section]
\newtheorem{propn}[thm]{Proposition}
\newtheorem{lemma}[thm]{Lemma}
\newtheorem{cor}[thm]{Corollary}
\newtheorem{predefn}[thm]{Definition}
\newenvironment{defn}{\begin{predefn}\upshape}{\end{predefn}}
\newtheorem{preexample}[thm]{Example}
\newenvironment{example}{\begin{preexample}\upshape}{\end{preexample}}
\newtheorem{preexamples}[thm]{Examples}
\newenvironment{examples}{\begin{preexamples}\upshape}{\end{preexamples}}
\newtheorem{preremark}[thm]{Remark}
\newenvironment{remark}{\begin{preremark}\upshape}{\end{preremark}}

\author{Tom Leinster%
\thanks{School of Mathematics, University of Edinburgh, Edinburgh
EH9 3JZ, UK; Tom.Leinster@ed.ac.uk.  
Partially supported by an EPSRC Advanced Research Fellowship.}}
\title{Codensity and the ultrafilter monad}
\date{\vspace*{-2ex}}

\begin{document}

\sloppy
\maketitle

\begin{abstract}
Even a functor without an adjoint induces a monad, namely, its codensity
monad; this is subject only to the existence of certain limits.  We clarify
the sense in which codensity monads act as substitutes for monads induced
by adjunctions.  We also expand on an undeservedly ignored theorem of
Kennison and Gildenhuys: that the codensity monad of the inclusion of
(finite sets) into (sets) is the ultrafilter monad.  This result is
analogous to the correspondence between measures and integrals.  So, for
example, we can speak of integration against an ultrafilter.  Using this
language, we show that the codensity monad of the inclusion of
(finite-dimensional vector spaces) into (vector spaces) is double
dualization.  From this it follows that compact Hausdorff spaces have a
linear analogue: linearly compact vector spaces.  Finally, we show that
ultra\emph{products} are categorically inevitable: the codensity monad of
the inclusion of (finite families of sets) into (families of sets) is the
ultraproduct monad.
\end{abstract}

\tableofcontents\

\begin{quote}
\emph{Now we have at last obtained permission to ventilate the facts\ldots}
\end{quote}
---Arthur Conan Doyle, The Adventure of the Creeping Man (1927)

\section*{Introduction}
\ucontents{section}{Introduction}

The codensity monad of a functor $G$ can be thought of as the monad induced
by $G$ and its left adjoint, even when no such adjoint exists.  We
explore the remarkable fact that when $G$ is the inclusion of the category
of finite sets into the category of all sets, the codensity monad of $G$ is
the ultrafilter monad.  Thus, the mere notion of finiteness of a set
gives rise automatically to the notion of ultrafilter, and so in turn to
the notion of compact Hausdorff space.

Many of the results in this paper are known, but not well known.  In
particular, the characterization of the ultrafilter monad as a codensity
monad appeared in the 1971 paper of Kennison and Gildenhuys~\cite{KeGi} and
the 1976 book of Manes (\cite{ManeAT}, Exercise~3.2.12(e)), but has not,
to my knowledge, appeared anywhere else.  Part of the purpose of this paper
is simply to ventilate the facts.

Ultrafilters belong to the minimalist world of set theory.  There are
several concepts in more structured branches of mathematics of which
ultrafilters are the set-theoretic shadow:
\begin{description}
\item[Probability measures]  An ultrafilter is a finitely additive
  probability measure in which every event has probability either $0$ or
  $1$ (Lemma~\ref{lemma:uf-prob}).  The elements of an ultrafilter on a set
  $X$ are the subsets that occupy `almost all' of $X$, and the other
  subsets of $X$ are to be regarded as `null', in the sense of measure
  theory. 

\item[Integration operators]  Ordinary real-valued integration on a measure
  space $(X, \mu)$ is an operation that takes as input a suitable function
  $f\from X \to \reals$ and produces as output an element $\int_X f
  \dee\mu$ of $\reals$.  We can integrate against ultrafilters, too.  Given
  an ultrafilter $\uf{U}$ on a set $X$, a set $R$, and a function
  $f\from X \to R$ with finite image, we obtain an element $\int_X f
  \dee\uf{U}$ of $R$; it is the unique element of $R$ whose $f$-fibre
  belongs to $\uf{U}$.

\item[Averages] To integrate a function against a \emph{probability}
  measure is to take its mean value with respect to that measure.
  Integrating against an ultrafilter $\uf{U}$ is more like taking the mode:
  if we think of elements of $\uf{U}$ as `large' then $\int_X f \dee\uf{U}$
  is the unique value of $f$ taken by a large number of elements of $X$.
  Ultrafilters are also used to prove results about more sophisticated
  types of average.  For example, a \demph{mean} on a group $G$ is a left
  invariant finitely additive probability measure defined on all subsets of
  $G$; a group is \demph{amenable} if it admits at least one mean.  Even to
  prove the amenability of $\integers$ is nontrivial, and is usually done by
  choosing a nonprincipal ultrafilter on $\nat$ (e.g.~\cite{Rund},
  Exercise~1.1.2). 

\item[Voting systems] In an election, each member of a set $X$ of voters
  chooses one element of a set $R$ of options.  A voting system computes
  from this a single element of $R$, intended to be some kind of average of
  the individual choices.  In the celebrated theorem of Arrow~\cite{Arro},
  $R$ has extra structure: it is the set of total orders on a list of
  candidates.  In our structureless context, ultrafilters can be seen as
  (unfair!)\ voting systems: when each member of a possibly-infinite set
  $X$ of voters chooses from a finite set $R$ of options, there
  is---according to any ultrafilter on $X$---a single option chosen by
  almost all voters, and that is the outcome of the election.
\end{description}

Section~\ref{sec:ufs} is a short introduction to ultrafilters.  It includes
a very simple and little-known characterization of ultrafilters, as
follows.  A standard lemma states that if $\uf{U}$ is an ultrafilter on a
set $X$, then whenever $X$ is partitioned into a finite number of (possibly
empty) subsets, exactly one belongs to $\uf{U}$.  But the converse is also
true~\cite{GaHo}: any set $\uf{U}$ of subsets of $X$ satisfying this
condition is an ultrafilter.  Indeed, it suffices to require this just for
partitions into three subsets.

We also review two characterizations of monads: one of
B\"orger~\cite{BoerCU}: 
\begin{ufchar}
  the ultrafilter monad is the terminal monad on $\Set$ that preserves
  finite coproducts
\end{ufchar}
and one of Manes~\cite{ManeTTC}:
\begin{ufchar}
  the ultrafilter monad is the monad for compact Hausdorff spaces.
\end{ufchar}

Density and codensity are reviewed in Section~\ref{sec:codensity}.  A
functor $G\from \cat{B} \to \cat{A}$ is either codense or not: yes or no.
Finer-grained information can be obtained by calculating the codensity
\emph{monad} of $G$.  This is a monad on $\cat{A}$, defined subject only to
the existence of certain limits, and it is the identity exactly when $G$ is
codense.  Thus, the codensity monad of a functor measures its failure to be
codense.

This prepares us for the codensity theorem of Kennison and Gildenhuys
(Section~\ref{sec:uf-co}): writing $\FinSet$ for the category of finite
sets,
\begin{ufchar}
  the ultrafilter monad is the codensity monad of the inclusion $\FinSet
  \incl \Set$.
\end{ufchar}
(In particular, since nontrivial ultrafilters exist, $\FinSet$ is not
codense in $\Set$.)  We actually prove a more general theorem, which has as
corollaries both this and an unpublished result of Lawvere.

Writing $\mnd{T} = (T, \eta, \mu)$ for the codensity monad of $\FinSet
\incl \Set$, the elements of $T(X)$ can be thought of as integration
operators on $X$, while the ultrafilters on $X$ are thought of as measures
on $X$.  The theorem of Kennison and Gildenhuys states that integration
operators correspond one-to-one with measures, as in analysis.  In general,
the notions of integration and codensity monad are bound together tightly.
This is one of our major themes.

Integration is most familiar when the integrands take values in some kind
of algebraic structure, such as $\reals$.  In Section~\ref{sec:rigs}, we
describe integration against an ultrafilter for functions taking values in
a rig (semiring).  We prove that when the rig $R$ is sufficiently
nontrivial, ultrafilters on $X$ correspond one-to-one with integration
operators for $R$-valued functions on $X$.

To continue, we need to review some further basic results on
codensity monads, including their construction as Kan extensions
(Section~\ref{sec:kan}).  This leads to another characterization: 
\begin{ufchar}
  the ultrafilter monad is the terminal monad on $\Set$ that restricts to
  the identity on $\FinSet$.
\end{ufchar}
In Section~\ref{sec:co-adj}, we justify the opening assertion of this
introduction: that the codensity monad of a functor $G$ is a surrogate for
the monad induced by $G$ and its left adjoint (which might not exist).  For
a start, if a left adjoint exists then the two monads are the same.  More
subtly, any monad on $\cat{A}$ induces a functor into $\cat{A}$ (the
forgetful functor on its category of algebras), and, under a completeness
hypothesis, any functor into $\cat{A}$ induces a monad on $\cat{A}$ (its
codensity monad).  Theorem~\ref{thm:co-adj}, due to Dubuc~\cite{Dubu},
states that the two processes are adjoint.  From this we deduce:
\begin{ufchar}
$\CptHff$ is the codomain of the universal functor from $\FinSet$ to\\
a category monadic over $\Set$.
\end{ufchar}
(This phrasing is slightly loose; see Corollary~\ref{cor:CptHff} for the
precise statement.)  Here $\CptHff$ is the category of compact Hausdorff
spaces.

We have seen that when standard categorical constructions are applied to
the inclusions $\FinSet \incl \Set$, we obtain the notions of ultrafilter
and compact Hausdorff space.  In Section~\ref{sec:dd} we ask what happens
when sets are replaced by vector spaces.  The answers give us the following
table of analogues:
\begin{center}
\begin{tabular}{ll}
sets                            &vector spaces  \\
finite sets                     &finite-dimensional vector spaces       \\
ultrafilters                    &elements of the double dual    \\
compact Hausdorff spaces        &linearly compact vector spaces.
\end{tabular}
\end{center}
The main results here are that the codensity monad of $\FDVect \incl \Vect$
is double dualization, and that its algebras are the linearly compact
vector spaces (defined below).  The close resemblance between
the $\Set$ and $\Vect$ cases raises the question: can analogous results be
proved for other algebraic theories?  We leave this open.  

It has long been a challenge to synthesize the complementary insights
offered by category theory and model theory.  For example, model theory
allows insights into parts of algebraic geometry where present-day category
theory seems to offer little.  (This is especially so when it comes to
transferring results between fields of positive characteristic and
characteristic zero, as exemplified by Ax's model-theoretic proof that
every injective endomorphism of a complex algebraic variety is
surjective~\cite{Ax}.)  A small part of this challenge is to find a
categorical home for the ultraproduct construction.

Section~\ref{sec:ups} does this.  The theorem of Kennison and Gildenhuys
shows that the notion of finiteness of a set leads inevitably to the notion
of ultrafilter.  Similarly, we show here that the notion of finiteness of a
\emph{family} of sets leads inevitably to the notion of ultraproduct.  More
specifically, we define a category of families of sets, and prove that the
codensity monad of the full subcategory of finitely-indexed families is the
ultraproduct monad.  This theorem (with a different proof) was transmitted
to me by the anonymous referee, to whom I am very grateful.

\paragraph*{History and related work}  The concept of density was first
isolated in a 1960 paper by Isbell~\cite{IsbeAS}, who gave a definition of
dense (or in his terminology, left adequate) full subcategory.  Ulmer
generalized the definition to arbitrary functors, not just inclusions of
full subcategories, and introduced the word `dense'~\cite{Ulme}.  At about
the same time, the codensity monad of a functor was defined by
Kock~\cite{KockCYR} (who gave it its name) and, independently, by Appelgate
and Tierney~\cite{ApTi} (who concentrated on the dual notion, calling it
the model-induced cotriple).

Other early sources on codensity monads are the papers of
Linton~\cite{LintOFS} and Dubuc~\cite{Dubu}.  (Co)density of functors is
covered in Chapter~X of Mac Lane's book~\cite{MacLCWM}, with codensity
monads appearing in the very last exercise.  Kelly's book~\cite{KellBCE}
treats (co)dense functors in detail, but omits (co)density (co)monads.

The codensity characterization of the ultrafilter monad seems to have first
appeared in the paper~\cite{KeGi} of Kennison and Gildenhuys, and is also
included as Exercise~3.2.12(e) of Manes's book~\cite{ManeAT}.  (Manes used
the term `algebraic completion' for codensity monad.)  It is curious that
no result resembling this appears in Isbell's 1960 paper, as even though he
did not have the notion of codensity monad available, he performed similar
and more set-theoretically sophisticated calculations.  However, his paper
does not mention ultrafilters.  On the other hand, a 2010 paper of Litt,
Abel and Kominers~\cite{LAK} proves a result equivalent to a weak form of
Kennison and Gildenhuys's theorem, but does not mention codensity.

The integral notation that we use so heavily has been used in similar ways
by Kock~\cite{KockCEQ,KockMEQ} and Lucyshyn-Wright~\cite{Lucy} (and
slightly differently by Lawvere and Rosebrugh in Chapter~8 of~\cite{LaRo}).
In~\cite{KockMEQ}, Kock traces the idea back to work of Linton and Wraith.

Richter~\cite{Rich} found a different proof of
Theorem~\ref{thm:Boerger-endo} below, originally due to B\"orger.
Section~3 of Kennison and Gildenhuys~\cite{KeGi} may provide some help in
answering the question posed at the end of Section~\ref{sec:dd}.

\paragraph*{Notation} We fix a category $\Set$ of sets
satisfying the axiom of choice.  $\Top$ is the category of all topological
spaces and continuous maps, and $\CAT$ is the category of locally small
categories.  When $X$ is a set and $Y$ is an object of some category, $[X,
  Y]$ denotes the $X$-power of $Y$, that is, the product of $X$ copies of
$Y$.  In particular, when $X$ and $Y$ are sets, $[X, Y]$ is the set $Y^X =
\Set(X, Y)$ of maps from $X$ to $Y$.  For categories $\cat{A}$ and
$\cat{B}$, we write $\ftrcat{\cat{A}}{\cat{B}}$ for the category of
functors from $\cat{A}$ to $\cat{B}$.  Where necessary, we silently assume
that our general categories $\cat{A}, \cat{B}, \ldots$ are locally small.

\section{Ultrafilters}
\label{sec:ufs}

We begin with the standard definitions.  Write $P(X)$ for the power set of a
set $X$.

\begin{defn}
Let $X$ be a set.  A \demph{filter} on $X$ is a subset $\uf{F}$ of $P(X)$ such
that:
\begin{enumerate}
\item $\uf{F}$ is upwards closed: if $Z \sub Y \sub X$ with $Z \in \uf{F}$
then $Y \in \uf{F}$
\item $\uf{F}$ is closed under finite intersections: $X \in \uf{F}$, and if
$Y, Z \in \uf{F}$ then $Y \cap Z \in \uf{F}$.
\end{enumerate}
\end{defn}

Filters on $X$ amount to meet-semilattice homomorphisms from $P(X)$ to
the two-element totally ordered set $2 = \{ 0 < 1 \}$, with
$f\from P(X) \to 2$ corresponding to the filter $f^{-1}(1) \sub X$.

It is helpful to view the elements of a filter as the `large' subsets of $X$,
and their complements as `small'.  Thus, the union of a finite number of small
sets is small.  An ultrafilter is a filter in which every subset is either
large or small, but not both.

\begin{defn}
Let $X$ be a set.  An \demph{ultrafilter} on $X$ is a filter $\uf{U}$ such
that for all $Y \sub X$, either $Y \in \uf{U}$ or $X \without Y \in \uf{U}$,
but not both.
\end{defn}

Ultrafilters on $X$ correspond to lattice homomorphisms $P(X) \to 2$. 

\begin{example}
Let $X$ be a set and $x \in X$.  The \demph{principal ultrafilter} on $x$
is the ultrafilter $\uf{U}_x = \{ Y \sub X \such x \in Y \}$.  Every
ultrafilter on a finite set is principal.
\end{example}

The set of filters on $X$ is ordered by inclusion.  The largest filter is
$P(X)$; every other filter is called \demph{proper}.  (What we call proper
filters are often just called filters.)  A standard lemma (Proposition~1.1
of~\cite{Eklo}) states that the ultrafilters are precisely the maximal
proper filters.  Zorn's lemma then implies that every proper filter is
contained in some ultrafilter.  No explicit example of a nonprincipal
ultrafilter can be given, since their existence implies a weak form of the
axiom of choice.  However:

\begin{example}
Let $X$ be an infinite set.  The subsets of $X$ with finite complement form
a proper filter $\uf{F}$ on $X$.  Then $\uf{F}$ is contained in some
ultrafilter, which cannot be principal.  Thus, every infinite set
admits at least one nonprincipal ultrafilter.
\end{example}

We will use the following simple characterization of ultrafilters.  The
equivalence of~(\ref{part:uf-concise-uf}) and~(\ref{part:uf-concise-n})
appears to be due to Galvin and Horn~\cite{GaHo}, whose result nearly
implies the equivalence with~(\ref{part:uf-concise-single}), too. 

\begin{propn}[Galvin and Horn]   \label{propn:uf-concise}
Let $X$ be a set and $\uf{U} \sub P(X)$.  The following are equivalent: 
\begin{enumerate}
\item   \label{part:uf-concise-uf}
$\uf{U}$ is an ultrafilter

\item   \label{part:uf-concise-n}
$\uf{U}$ satisfies the \demph{partition condition:} for all $n \geq 0$
and partitions 
\[
X = Y_1 \amalg \cdots \amalg Y_n
\]
of $X$ into $n$ pairwise disjoint (possibly empty) subsets, there
is exactly one $i \in \{1, \ldots, n\}$ such that $Y_i \in \uf{U}$.
\setcounter{bean}{\value{enumi}}
\end{enumerate}
Moreover, for any $N \geq 3$, these conditions are equivalent to:
\begin{enumerate}
\setcounter{enumi}{\value{bean}}
\item   \label{part:uf-concise-single}
$\uf{U}$ satisfies the partition condition for $n = N$.
\end{enumerate}
\end{propn}

\begin{proof}
Let $N \geq 3$.  The implication
(\ref{part:uf-concise-uf})$\implies$(\ref{part:uf-concise-n}) is standard, 
and (\ref{part:uf-concise-n})$\implies$(\ref{part:uf-concise-single}) is
trivial.  Now assume~(\ref{part:uf-concise-single}); we
prove~(\ref{part:uf-concise-uf}).

From the partition $X = X \amalg \underbrace{\emptyset \amalg \cdots \amalg
\emptyset}_{N - 1}$ and the fact that $N \geq 3$, we deduce that
$\emptyset \not\in \uf{U}$ and $X \in \uf{U}$.  It follows that $\uf{U}$
satisfies the partition condition for all $n \leq N$.  Taking $n = 2$, this
implies that for all $Y \sub X$, either $Y \in \uf{U}$ or $X\without Y \in
\uf{U}$, but not both.  It remains to prove that $\uf{U}$ is upwards closed
and closed under binary intersections.

For upwards closure, let $Z \sub Y \sub X$ with $Z \in \uf{U}$.  We have
\[
X = Z \amalg (Y \without Z) \amalg (X \without Y)
\]
with $Z \in \uf{U}$, so $X \without Y \not\in \uf{U}$.  Hence $Y \in \uf{U}$.

To prove closure under binary intersections, first note that if $Y_1, Y_2
\in \uf{U}$ then $Y_1 \cap Y_2 \neq \emptyset$: for if $Y_1 \cap Y_2 =
\emptyset$ then $Y_1 \sub X \without Y_2$, so $X \without Y_2 \in \uf{U}$
by upwards closure, so $Y_2 \not\in \uf{U}$, a contradiction.  Now let $Y,
Z \in \uf{U}$ and consider the partition
\[
X = (Y \cap Z) \amalg (Y \without Z) \amalg (X \without Y).
\]
Exactly one of these three subsets, say $S$, is in $\uf{U}$.  But $S, Y \in
\uf{U}$, so $S \cap Y \neq \emptyset$, so $S \neq X \without Y$; similarly,
$S \neq Y \without Z$.
Hence $S = Y \cap Z$, as required.
\done
\end{proof}

Perhaps the most striking part of this result is:
\begin{cor}     \label{cor:uf-3}
Let $X$ be a set and $\uf{U}$ a set of subsets of $X$ such that whenever
$X$ is expressed as a disjoint union of three subsets, exactly one belongs
to $\uf{U}$.  Then $\uf{U}$ is an ultrafilter.
\done
\end{cor}

The number three cannot be lowered to two: consider a three-element set $X$
and the set $\uf{U}$ of subsets with at least two elements.

Given a map of sets $f\from X \to X'$ and a filter $\uf{F}$ on $X$, there
is an induced filter
\[
f_* \uf{F}
=
\{ Y' \sub X' \such f^{-1}Y' \in \uf{F} \}
\]
on $X'$.  If $\uf{F}$ is an ultrafilter then so is $f_* \uf{F}$.  This
defines a functor
\[
U\from \Set \to \Set
\]
in which $U(X)$ is the set of ultrafilters on $X$.

In fact, $U$ carries the structure of a monad, $\mnd{U}$.  The unit map
$X \to U(X)$ sends $x \in X$ to the principal ultrafilter
$\uf{U}_x$.  We will avoid writing down the multiplication explicitly.
(The contravariant power set functor $P$ from $\Set$ to $\Set$ is
self-adjoint on the right, and therefore induces a monad $PP$ on $\Set$; it
contains $\mnd{U}$ as a submonad.)  What excuses us from this duty is the
following powerful pair of results, both due to B\"orger~\cite{BoerCU}.

\begin{thm}[B\"orger]   \label{thm:Boerger-endo}
The ultrafilter endofunctor $U$ is terminal among all endofunctors of $\Set$
that preserve finite coproducts.  
\end{thm}

\begin{prooflike}{Sketch proof}
Given a finite-coproduct-preserving endofunctor $S$ of $\Set$, the unique
natural transformation $\alpha\from S \to U$ is described as follows: for
each set $X$ and element $\sigma \in S(X)$,
\[
\alpha_X(\sigma)
=
\{ Y \sub X \such \sigma \in \im(S(Y \incl X)) \}.
\]
For details, see Theorem~2.1 of~\cite{BoerCU}.
\done
\end{prooflike}

\begin{cor}[B\"orger]   \label{cor:Boerger-monad}
The ultrafilter endofunctor $U$ has a unique monad structure.  With this
structure, it is terminal among all finite-coproduct-preserving monads on $\Set$.
\end{cor}

\begin{proof}
(Corollary~2.3 of~\cite{BoerCU}.)  Since $U \of U$ and the identity
preserve finite coproducts, there are unique natural transformations $U \of U
\to U$ and $1 \to U$.  The monad axioms follow by terminality of the
endofunctor $U$, as does terminality of the monad.  
\done
\end{proof}

There is also a topological description of the ultrafilter monad.
As shown by Manes~\cite{ManeTTC}, it is the monad induced by the forgetful
functor $\CptHff \to \Set$ and its left adjoint.  In particular, the
Stone--\v{C}ech compactification of a discrete space is the set of
ultrafilters on it.

\section{Codensity}
\label{sec:codensity}

Here we review the definitions of codense functor and codensity monad.  The
dual notion, density, has historically been more prominent, so we begin our
review there.

As shown by Kan, any functor $F$ from a small category $\cat{A}$ to a
cocomplete category $\cat{B}$ induces an adjunction
\[
\xymatrix@C+5ex{
\cat{B}
\ar@<1ex>[r]^-{\Hom(F, \dashbk)} \ar@{}[r]|-\top &
\ftrcat{\cat{A}^\op}{\Set}
\ar@<1ex>[l]^-{\dashbk \otimes F}
}
\]
where $(\Hom(F,B))(A) = \cat{B}(F(A), B)$.  A famous example is the functor
$F\from \Delta \to \Top$ assigning to each nonempty finite ordinal $[n]$
the topological $n$-simplex $\Delta^n$.  Then $\Hom(F, \dashbk)$ sends a
topological space to its singular simplicial set, and $\dashbk \otimes F$
sends a simplicial set to its geometric realization.

Another example gives an abstract explanation of the concept of sheaf
(\cite{MaMo}, Section~II.6).  Let $X$ be a topological space, with poset
$\Open(X)$ of open subsets.  Define $F \from \Open(X) \to \Top/X$ by $F(W)
= (W \incl X)$.  This induces an adjunction between presheaves on $X$ and
spaces over $X$, and, like any adjunction, it restricts canonically to an
equivalence between full subcategories.  Here, these are the categories of
sheaves on $X$ and \'etale bundles over $X$.  The induced monad on the
category of presheaves is sheafification.

In general, $F$ is \demph{dense} if the right adjoint $\Hom(F, \dashbk)$ is
full and faithful, or equivalently if the counit is an isomorphism.  For
the counit to be an isomorphism means that every object of $\cat{B}$ is a
colimit of objects of the form $F(A)$ ($A \in \cat{A}$) in a canonical way;
for example, the Yoneda embedding $\cat{A} \to \ftrcat{\cat{A}^\op}{\Set}$
is dense, so every presheaf is canonically a colimit of representables.
More loosely, $F$ is dense if the objects of $\cat{B}$ can be effectively
probed by mapping into them from objects of the form $F(A)$.  In the case
of the Yoneda embedding, this is the familiar idea that presheaves can be
probed by mapping into them from representables.

Finitely presentable objects provide further important examples.  For
instance, the embedding $\Grp_\text{fp} \incl \Grp$ is dense, where $\Grp$
is the category of groups and $\Grp_\text{fp}$ is the full subcategory of
groups that are finitely presentable.  Similarly, $\FinSet$ is dense in
$\Set$.

Here we are concerned with \emph{co}density.  The general theory is of course
formally dual to that of density, but its application to familiar functors
seems not to have been so thoroughly explored.

Let $G\from \cat{B} \to \cat{A}$ be a functor.  There is an induced functor
\[
\Hom(\dashbk, G)\from \cat{A} \to \ftrcat{\cat{B}}{\Set}^\op
\]
defined by
\[
\bigl( \Hom(A, G) \bigr)(B)
=
\cat{A}(A, G(B))
\]
($A \in \cat{A}, B \in \cat{B}$).  The functor $G$ is \demph{codense} if
$\Hom(\dashbk, G)$ is full and faithful.  

Assume for the rest of this section that $\cat{B}$ is \demph{essentially
small} (equivalent to a small category) and that $\cat{A}$ has small
limits.  (This assumption will be relaxed in Section~\ref{sec:kan}.)  Then
$\Hom(\dashbk, G)$ has a right adjoint, also denoted by $\Hom(\dashbk, G)$:
\begin{equation}        \label{eq:coden-adjn}
\xymatrix@C+5ex{
\cat{A}
\ar@<1ex>[r]^-{\Hom(\dashbk, G)} \ar@{}[r]|-\bot &
\ftrcat{\cat{B}}{\Set}^\op.
\ar@<1ex>[l]^-{\Hom(\dashbk, G)}
}
\end{equation}
This right adjoint can be described as an end or as a limit: for $Y \in
\ftrcat{\cat{B}}{\Set}$, 
\[
\Hom(Y, G)
=
\int_{B \in \cat{B}} [Y(B), G(B)]
=
\lt{B \in \cat{B},\ y \in Y(B)} G(B),
\]
where the square bracket notation is as defined at the end of the
introduction, and the limit is over the category of elements of $Y$.  If
$\cat{A} = \Set$ then $\Hom(Y, G)$ is the set of natural transformations
from $Y$ to $G$.  In any case, the adjointness asserts that
\[
\cat{A}(A, \Hom(Y, G))
\iso
\ftrcat{\cat{B}}{\Set} (Y, \Hom(A, G))
\]
naturally in $A \in \cat{A}$ and $Y \in \ftrcat{\cat{B}}{\Set}$.  

The adjunction~\eqref{eq:coden-adjn} induces a monad $\mnd{T}^G = (T^G,
\eta^G, \mu^G)$ on $\cat{A}$, the \demph{codensity monad} of $G$.
Explicitly,
\[
T^G(A) 
=
\int_{B \in \cat{B}} [\cat{A}(A, G(B)),\, G(B)]
=
\lt{\substack{B \in \cat{B},\\ f \from A \to G(B)}}
G(B)
\]
($A \in \cat{A}$).  As for any adjunction, the left adjoint is full and
faithful if and only if the unit is an isomorphism.  Thus, $G$ is codense
if and only if for each $A \in \cat{A}$, the canonical map
\[
\eta^G_A \from A \to \int_B [\cat{A}(A, G(B)),\, G(B)]
\]
is an isomorphism.  (Then each object of $\cat{A}$ is a limit of objects
$G(B)$ in a canonical way.)  This happens if and only if the codensity monad
of $G$ is isomorphic to the identity.  In that sense, the codensity monad
of a functor measures its failure to be codense.

In many cases of interest, $G$ is a subcategory inclusion $\cat{B} \incl
\cat{A}$.  We then transfer epithets, calling $\cat{B}$ codense if $G$ is,
and writing $\mnd{T}^\cat{B}$ instead of $\mnd{T}^G$.

We continue with the theory of codensity monads in Sections~\ref{sec:kan}
and~\ref{sec:co-adj}, but we now have all we need to proceed to the result
on ultrafilters.

\section{Ultrafilters via codensity}
\label{sec:uf-co}

Here we give an account of the fact, due to Kennison and Gildenhuys, that
the ultrafilter monad is the codensity monad of the subcategory $\FinSet$
of $\Set$.  The proof is made more transparent by adopting the language of
integration and measure.

First, though, let us see roughly why the result might be true.  Write
$\mnd{T} = (T, \eta, \mu)$ for the codensity monad of $\FinSet \incl \Set$.
Fix a set $X$.  Then
\[
T(X) = \int_{B \in \FinSet} [[X, B], B],
\]
which is the set of natural transformations
\[
\xymatrix@R+1em{
\FinSet \rtwocell^{[X, \dashbk]}_{\text{inclusion}} &
\Set.\\
}
\]
An element of $T(X)$ is, therefore, an operation that takes as input a finite
set $B$ and a function $X \to B$, and returns as output an element of $B$;
and it does so in a way that is natural in $B$.  There is certainly one
such operation for each element $x$ of $X$, namely, evaluation at $x$.  Less
obviously, there is one such operation for each ultrafilter $\uf{U}$ on
$X$: given $f\from X \to B$ as input, return as output the unique
element $b \in B$ such that $f^{-1}(b) \in \uf{U}$.  (There \emph{is} a
unique $b$ with this property, by
Proposition~\ref{propn:uf-concise}(\ref{part:uf-concise-n}).)  For example, if
$\uf{U}$ is the principal ultrafilter on $x \in X$, this operation is just
evaluation at $x$.  It turns out that every element $I \in T(X)$ arises from
an ultrafilter, which one recovers from $I$ by taking $B = 2$ and noting
that $[[X, 2], 2] \iso PP(X)$.  That, in essence, is how we will prove
the theorem.  

An ultrafilter is a probability measure that paints the world in black and
white: everything is either almost surely true or almost surely false.
Indeed, an ultrafilter $\uf{U}$ on a set $X$ is in particular a subset of
$P(X)$, and therefore has a characteristic function $\mu_\uf{U}\from P(X) \to
\{0, 1\}$.  On the other hand, a \demph{finitely additive measure} on a set
$X$ (or properly speaking, on the algebra of all subsets of $X$) is a
function $\mu\from P(X) \to [0, \infty]$ such that
\[
\mu(\emptyset) = 0,
\qquad
\mu(Y \cup Z) + \mu(Y \cap Z) = \mu(Y) + \mu(Z)
\]
for all $Y, Z \sub X$.  (Equivalently, $\mu(\bigcup_i Y_i) = \sum_i
\mu(Y_i)$ for all finite families $(Y_i)$ of pairwise disjoint subsets of
$X$.)  We call $\mu$ a finitely additive \demph{probability} measure if
also $\mu(X) = 1$.  The following correspondence has been observed many
times.

\begin{lemma}   \label{lemma:uf-prob}
Let $X$ be a set.  A subset $\uf{U}$ of $P(X)$ is an ultrafilter if and only
if its characteristic function $\mu_\uf{U}\from P(X) \to \{0, 1\}$ is a
finitely additive probability measure.  This defines a bijection
between the ultrafilters on $X$ and the finitely additive probability
measures on $X$ with values in $\{0, 1\}$.  \done
\end{lemma}

With every notion of measure comes a notion of integration.  Integrating a
function with respect to a \emph{probability} measure amounts to taking its
average value, and taking averages typically requires some algebraic
or order-theoretic structure, which we do not have.  Nevertheless,
it can be done, as follows.

Let us say that a function between sets is \demph{simple} if its image is
finite.  (The name is justified in Section~\ref{sec:rigs}.)  The set of
simple functions from one set, $X$, to another, $R$, is written as $\Simp(X,
R)$; categorically, it is the coend
\[
\Simp(X, R) 
=
\int^{B \in \FinSet} \Set(X, B) \times \Set(B, R).
\]
The next result states that given an ultrafilter $\uf{U}$ on a set $X$,
there is a unique sensible way to define integration of simple functions on
$X$ with respect to the measure $\mu_\uf{U}$.  The two conditions defining
`sensible' are that the average value (integral) of a constant function is
that constant, and that changing a function on a set of measure zero does
not change its integral.  

\begin{propn}   \label{propn:uf-int-unique}
Let $X$ be a set and $\uf{U}$ an ultrafilter on $X$.  Then for each set
$R$, there is a unique map
\[
\int_X \idashbk \dee\uf{U}\from \Simp(X, R) \to R
\]
such that
\begin{enumerate}
\item   \label{item:uf-int-const}
$\int_X r \dee\uf{U} = r$ for all $r \in R$, where the integrand is the
function with constant value $r$

\item   \label{item:uf-int-zero}
$\int_X f \dee\uf{U} = \int_X g \dee\uf{U}$ whenever $f, g \in \Simp(X, R)$
with $\{x \in X \such f(x) = g(x) \} \in \uf{U}$.
\end{enumerate}
\end{propn}

In analysis, it is customary to write $\int_X f \dee\mu$ for the integral
of a function $f$ with respect to (or `against') a measure $\mu$.  Logically,
then, we should write our integration operator as $\int_X \idashbk
\dee\mu_\uf{U}$.  However, we blur the distinction between $\uf{U}$ and
$\mu_\uf{U}$, writing $\int_X \idashbk \dee\uf{U}$ (or just $\int \idashbk
\dee\uf{U}$) instead.

\begin{proof}
Let $R$ be a set.  For existence, given any $f \in \Simp(X, R)$, simplicity
guarantees that there is a unique element $\int_X f \dee\uf{U}$ of $R$ such
that
\[
f^{-1} \biggl( \int_X f \dee\uf{U} \biggr) \in \uf{U}.
\]
Condition~\eqref{item:uf-int-const} holds because $X \in \uf{U}$.
For~\eqref{item:uf-int-zero}, let $f$ and $g$ be simple functions such that
$\Eq(f, g) = \{ x \in X \such f(x) = g(x) \}$ belongs to $\uf{U}$.  We have
\[
f^{-1} \biggl( \int_X f \dee\uf{U} \biggr) \cap \Eq(f, g)
\,\sub\, 
g^{-1} \biggl( \int_X f \dee\uf{U} \biggr),
\]
and $f^{-1}\bigl( \int f \dee\uf{U} \bigr), \Eq(f, g) \in \uf{U}$, so by
definition of ultrafilter, $g^{-1}\bigl( \int f \dee\uf{U} \bigr) \in
\uf{U}$.  But $\int g \dee\uf{U}$ is by definition the unique element $r$
of $R$ such that $g^{-1}(r) \in \uf{U}$, so $\int f \dee\uf{U} = \int g
\dee\uf{U}$, as required.

For uniqueness, let $f \in \Simp(X, R)$.  Since $f$ is simple, there is a
unique $r \in R$ such that $f^{-1}(r) \in \uf{U}$.  Then $\Eq(f, r) \in
\uf{U}$, so~\eqref{item:uf-int-const} and~\eqref{item:uf-int-zero} force
$\int f \dee\uf{U} = r$.  \done
\end{proof}

Integration is natural in both the codomain $R$ and the domain pair $(X,
\uf{U})$: 
\begin{lemma}   \label{lemma:uf-int-nat}
\begin{enumerate}
\item   \label{part:uf-int-nat-cod}
Let $\uf{U}$ be an ultrafilter on a set $X$.  Then integration of simple
functions against $\uf{U}$ defines a natural transformation
\[
\xymatrix@C+10ex{
\Set \rtwocell^{\Simp(X, \dashbk)}_{\id}
{\quad\quad\int \idashbk \dee\uf{U}} &
\Set.
}
\]

\item   \label{part:uf-int-nat-dom}
For any map $X \toby{p} Y$ of sets and ultrafilter $\uf{U}$ on $X$, the
triangle
\[
\xymatrix{
\Simp(X, \dashbk) \ar[rd]_{\int_X \idashbk \dee\uf{U}}  &
        &
\Simp(Y, \dashbk) \ar[ll]_{\dashbk\of p} 
\ar[ld]^{\int_Y \idashbk \dee(p_* \uf{U})}      \\
        &
\id
}
\]
in $\ftrcat{\Set}{\Set}$ commutes.
\end{enumerate} 
\end{lemma}

\begin{proof}
For~\eqref{part:uf-int-nat-cod}, we must prove that for any map $R
\toby{\theta} S$ of finite sets and any function $f\from X \to R$, 
\begin{equation}        \label{eq:uf-int-nat-cod-pf}
\theta \biggl( \int_X f \dee\uf{U} \biggr)
=
\int_X \theta \of f \dee\uf{U}.
\end{equation}
Indeed,
\[
(\theta \of f)^{-1} \biggl( \theta \biggl( \int_X f \dee\uf{U} \biggr) \biggr)
\supseteq
f^{-1} \biggl( \int_X f \dee\uf{U} \biggr) 
\in 
\uf{U},
\]
so $(\theta \of f)^{-1} \bigl( \theta \bigl( \int f \dee\uf{U} \bigr)
\bigr) \in \uf{U}$, and~\eqref{eq:uf-int-nat-cod-pf} follows.

For~\eqref{part:uf-int-nat-dom}, let $R \in \FinSet$ and $g \in \Simp(Y,
R)$.  We must prove that
\begin{equation}        \label{eq:uf-int-nat-dom-pf}
\int_X (g \of p) \dee\uf{U}
=
\int_Y g \dee(p_* \uf{U})
\end{equation}
(the analogue of the classical formula for integration under a change of
variable).  Indeed,
\[
g^{-1} \biggl( \int_Y g \dee(p_* \uf{U}) \biggr) \in p_* \uf{U},
\]
which by definition of $p_* \uf{U}$ means that
\[
(g \of p)^{-1} \biggl( \int_Y g \dee(p_* \uf{U}) \biggr) 
\in 
\uf{U},
\]
giving~\eqref{eq:uf-int-nat-dom-pf}.
\done
\end{proof}

For the next few results, we will allow $R$ to vary within a subcategory
$\cat{B}$ of $\FinSet$.  (The most important case is $\cat{B} = \FinSet$.)
Clearly $\Simp(X, B) = [X, B]$ for all $B \in \cat{B}$.  The notation
$\mnd{T}^\cat{B}$ will mean the codensity monad of $\cat{B} \incl \Set$
(not $\cat{B} \incl \FinSet$).  Thus, whenever $X$ is a set, $T^\cat{B}(X)$
is the set of natural transformations
\[
\xymatrix@C+1em{
\cat{B}
\rtwocell^{[X, \dashbk]}_{\text{inclusion}}{}   &
\Set.
}
\]
We will regard elements of $T^\cat{B}(X)$ as integration operators: an
element $I \in T^\cat{B}(X)$ consists of a function $I = I_B\from [X, B]
\to B$ for each $B \in \cat{B}$, such that
\begin{equation}        \label{eq:int-nat-square}
\begin{centerdiag}
\xymatrix{
[X, B] \ar[r]^{\theta\of\dashbk} \ar[d]_{I_B}   &
[X, C] \ar[d]^{I_C}     \\
B \ar[r]_\theta &
C
}
\end{centerdiag}
\end{equation}
commutes whenever $B \toby{\theta} C$ is a map in $\cat{B}$.

\begin{propn}   \label{propn:uf-transf}
Let $\cat{B}$ be a subcategory of $\FinSet$.  Then there is a natural
transformation $U \to T^\cat{B}$ with components
\begin{equation}        \label{eq:uf-int-transf}
\begin{mapdefn}
U(X)    &\to            &T^\cat{B}(X)                   \\
\uf{U}  &\longmapsto    &\int_X \idashbk \dee\uf{U}  
\end{mapdefn}
\end{equation}
($X \in \Set$).
\end{propn}

\begin{proof}
Lemma~\ref{lemma:uf-int-nat}(\ref{part:uf-int-nat-cod}) guarantees
that~\eqref{eq:uf-int-transf} is a well-defined function for each $X$.   
Lemma~\ref{lemma:uf-int-nat}(\ref{part:uf-int-nat-dom}) tells us that it is
natural in $X$.
\done
\end{proof}

The transformation of Proposition~\ref{propn:uf-transf} turns measures
(ultrafilters) into integration operators.  In analysis, we recover a
measure $\mu$ from its corresponding integration operator via the equation
$\mu(Y) = \int \chi_Y \dee\mu$.  To imitate this here, we need some notion
of characteristic function, and for that we need $\cat{B}$ to contain some
set with at least two elements.

So, suppose that we have fixed some set $\Omega \in \cat{B}$ and elements
$0, 1 \in \Omega$ with $0 \neq 1$.  For any set $X$ and $Y \sub X$, define
$\chi_Y \from X \to \Omega$ by
\begin{equation}        \label{eq:char-fn-defn}
\chi_Y(x) 
=
\begin{cases}
1       &\text{if } x \in Y             \\
0       &\text{otherwise.}
\end{cases}
\end{equation}
Then for any ultrafilter $\uf{U}$ on $X$, we have
\begin{equation}        \label{eq:char-fn-int}
\int_X \chi_Y \dee\uf{U}
=
\begin{cases}
1       &\text{if } Y \in \uf{U}        \\
0       &\text{otherwise.}
\end{cases}
\end{equation}
Hence
\begin{equation}        \label{eq:recover-uf}
\uf{U} 
=
\biggl\{ Y \sub X \such \int_X \chi_Y \dee\uf{U} = 1 \biggr\}.
\end{equation}
We have thus recovered $\uf{U}$ from $\int_X \idashbk \dee\uf{U}$.

The full theorem is as follows.

\begin{thm}     \label{thm:main}
Let $\cat{B}$ be a full subcategory of $\FinSet$ containing at least one
set with at least three elements.  Then the codensity monad of $\cat{B}
\incl \Set$ is isomorphic to the ultrafilter monad.
\end{thm}

\begin{proof}
We show that the natural transformation $U \to T^\cat{B}$ of
Proposition~\ref{propn:uf-transf} is a natural isomorphism.  Then by
Corollary~\ref{cor:Boerger-monad}, it is an isomorphism of monads.

Let $X$ be a set and $I \in T^\cat{B}(X)$.  We must show that there is a
unique ultrafilter $\uf{U}$ on $X$ such that $I = \int_X \idashbk
\dee\uf{U}$.  Choose a set $\Omega \in \cat{B}$ with at least two elements,
say $0$ and $1$, and whenever $Y \sub X$, define $\chi_Y$ as
in~\eqref{eq:char-fn-defn}. 

Uniqueness follows from~\eqref{eq:recover-uf}.  For existence, put $\uf{U}
= \{ Y \sub X \such I(\chi_Y) = 1 \}$.  Whenever $B$ is a set in $\cat{B}$
and $f\from X \to B$ is a function, $I(f)$ is the unique element of $B$
satisfying $f^{-1}(I(f)) \in \uf{U}$: for given $b \in B$, we have
\begin{align*}
f^{-1}(b) \in \uf{U}    &
\iff
I(\chi_{f^{-1}(b)}) = 1
\iff
I(\chi_{\{b\}} \of f) = 1
\iff
\chi_{\{b\}}(I(f)) = 1  \\
        &
\iff
b = I(f),
\end{align*}
where the penultimate step is by~\eqref{eq:int-nat-square}.  Applying this
when $B$ is a set in $\cat{B}$ with at least three elements proves that
$\uf{U}$ is an ultrafilter, by
Proposition~\ref{propn:uf-concise}(\ref{part:uf-concise-single}).
Moreover, since $f^{-1}(I(f)) \in \uf{U}$ for any $f$, we have $I = \int
\idashbk \dee\uf{U}$, as required.  
\done
\end{proof}

\begin{remark}  \label{rmk:monad-strategies-set}
In this proof, we used B\"orger's Corollary~\ref{cor:Boerger-monad} 
as a labour-saving device; it excused us from checking that the constructed
isomorphism $U \to T^\cat{B}$ preserves the monad structure.  We could
also have checked this directly.  Remark~\ref{rmk:monad-strategies-vs}
describes a third method.
\end{remark}

\begin{remark}
The condition that $\cat{B}$ contains at least one set with at least three
elements is sharp.  There are $2^3 = 8$ full subcategories $\cat{B}$ of
$\Set$ containing only sets of cardinality $0$, $1$ or $2$, and in no case
is $\mnd{T}^\cat{B}$ isomorphic to the ultrafilter monad.  If $2 \not\in
\cat{B}$ then $T^\cat{B}(X) = 1$ for all nonempty $X$.  If $2 \in \cat{B}$
then $T^\cat{B}(X)$ is canonically isomorphic to the set of all $\uf{U}
\sub P(X)$ satisfying the partition condition of
Proposition~\ref{propn:uf-concise} for $n \in \{1, 2\}$.  In that case,
$U(X) \sub T^\cat{B}(X)$, but by the example after Corollary~\ref{cor:uf-3},
the inclusion is in general strict.
\end{remark}

We immediately deduce an important result from~\cite{KeGi}:

\begin{cor}[Kennison and Gildenhuys]      
The codensity monad of $\FinSet \incl \Set$ is the ultrafilter monad.
\done
\end{cor}

We can also deduce an unpublished result stated by Lawvere in
2000~\cite{LawvTTL}.  (See also~\cite{nLabUF}.)  It does not mention
codensity explicitly.  Write $\End(B)$ for the endomorphism monoid of a set
$B$, and $\Set^{\End(B)}$ for the category of left $\End(B)$-sets.  Given a
set $X$, equip $[X, B]$ with the natural left action by $\End(B)$.

\begin{cor}[Lawvere]    \label{cor:lawvere}
Let $B$ be a finite set with at least three elements.  Then
\[
\Set^{\End(B)} ([X, B], B) \iso U(X)
\]
naturally in $X \in \Set$.
\end{cor}

\begin{proof}
Let $\cat{B}$ be the full subcategory of $\Set$ consisting of the single
object $B$.  Then $T^\cat{B}(X) = \Set^{\End(B)}([X, B], B)$, and the result
follows from Theorem~\ref{thm:main}.
\done
\end{proof}

For example, let $3$ denote the three-element set; then an ultrafilter on $X$
amounts to a map $3^X \to 3$ respecting the natural action of the
27-element monoid $\End(3)$.

We have exploited the idea that an ultrafilter on a set $X$ is a primitive
sort of probability measure on $X$.  But there are monads other than
$\mnd{U}$, in other settings, that assign to a space $X$ some space of
measures on $X$: for instance, there are those of Giry~\cite{Giry} and
Lucyshyn-Wright~\cite{Lucy}.  It may be worth investigating whether they,
too, arise canonically as codensity monads.

\section{Integration of functions taking values in a rig}
\label{sec:rigs}

Integration of the most familiar kind involves integrands taking values in
the ring $\reals$ and an integration operator that is $\reals$-linear.  So
far, the codomains of our integrands have been mere sets.  However, we can
say more when the codomain has algebraic structure.  The resulting theory
sheds light on the relationship between integration as classically
understood and integration against an ultrafilter.

Let $R$ be a rig (semiring).  To avoid complications, we take all rigs to
be commutative.  Since $R$ has elements $0$ and $1$, we may define the
characteristic function $\chi_Y\from X \to R$ of any subset $Y$ of a set
$X$, as in equation~\eqref{eq:char-fn-defn}.

In analysis, a function on a measure space $X$ is called simple if it is
a finite linear combination of characteristic functions of measurable
subsets of $X$.  The following lemma justifies our own use of the word.

\begin{lemma}   \label{lemma:simple}
A function from a set $X$ to a rig $R$ is simple if and only if it is a
finite $R$-linear combination of characteristic functions of subsets of
$X$. 
\done
\end{lemma}

Integration against an ultrafilter is automatically linear:

\begin{lemma}   \label{lemma:uf-int-linear}
Let $X$ be a set, $\uf{U}$ an ultrafilter on $X$, and $R$ a rig.  Then the
map $\int_X \idashbk \dee\uf{U} \from \Simp(X, R) \to R$ is $R$-linear.
\end{lemma}

Here, we are implicitly using the notion of a \demph{module} over a rig $R$,
which is an (additive) commutative monoid equipped with an action by $R$
satisfying the evident axioms.  In particular, $\Simp(X, R)$ is an $R$-module
with pointwise operations.

\begin{proof}
We have the natural transformation
\[
\xymatrix@C+10ex{
\Set
\rtwocell^{\Simp(X, \dashbk)}_{\id}
{\quad\quad \int\idashbk\dee\uf{U}}     &
\Set
}
\]
in which $\Set$ has finite products and both functors preserve finite
products.  The theory of $R$-modules is a finite product theory, so taking
internal $R$-modules throughout gives a natural transformation
\[
\xymatrix@C+10ex{
R\hyph\Mod
\rtwocell^{\Simp(X, \dashbk)}_{\id}%
{\quad\quad \int\idashbk\dee\uf{U}}     &
R\hyph\Mod
}
\]
This new functor $\Simp(X, \dashbk)$ sends an $R$-module $M$ to $\Simp(X,
M)$ with the pointwise $R$-module structure, and $\int \idashbk
\dee\uf{U}$ defines an $R$-linear map $\Simp(X, M) \to M$.  Applying this
to $M = R$ gives the result.
\done
\end{proof}

\begin{propn}
Let $X$ be a set, $\uf{U}$ an ultrafilter on $X$, and $R$ a rig.  Then
$\int_X \idashbk \dee\uf{U}$ is the unique $R$-linear map $\Simp(X, R) \to
R$ such that for all $Y \sub X$,
\[
\int_X \chi_Y \dee\uf{U}
=
\begin{cases}
1       &\text{if } Y \in \uf{U}        \\
0       &\text{otherwise}
\end{cases}
\]
(that is, $\int_X \chi_Y \dee\uf{U} = \mu_\uf{U}(Y)$).
\end{propn}

\begin{proof}
We have already shown that $\int_X \idashbk \dee\uf{U}$ has the desired
properties (Lemma~\ref{lemma:uf-int-linear} and
equation~\eqref{eq:char-fn-int}).  Uniqueness follows from
Lemma~\ref{lemma:simple}.
\done
\end{proof}

Let $X$ be a set and $R$ a rig.  For any ultrafilter $\uf{U}$ on $X$, the
$R$-linear map $\int \idashbk \dee\uf{U} \from \Simp(X, R) \to R$ has the
property that $\int f \dee\uf{U}$ always belongs to $\im(f)$.  Abstracting,
let us define an \demph{$R$-valued integral on $X$} to be an $R$-linear map
$I\from \Simp(X, R) \to R$ such that $I(f) \in \im(f)$ for all $f \in
\Simp(X, R)$.

Our main result states that an ultrafilter on a set $X$ is essentially the
same thing as an $R$-valued integral on $X$, as long as the rig $R$ is
sufficiently nontrivial.  

\begin{thm}     \label{thm:rig-main}
Let $R$ be a rig in which $3 \neq 1$.  Then for any set $X$, there is a
canonical bijection
\[
U(X) \toby{\sim} \{ R \text{-valued integrals on } X \},
\]
defined by $\uf{U} \mapsto \int_X \idashbk \dee\uf{U}$.
\end{thm}

\begin{proof}
Injectivity follows from the equation
\[
\uf{U} = \biggl\{ Y \sub X \such \int_X \chi_Y \dee\uf{U} = 1 \biggr\}
\]
($\uf{U} \in U(X)$), which is itself a consequence of~\eqref{eq:char-fn-int}
and the fact that $0 \neq 1$ in $R$.

For surjectivity, let $I$ be an $R$-valued integral on $X$.  Put $\uf{U}
= \{ Y \sub X \such I(\chi_Y) = 1 \}$.  To show that $\uf{U}$ is an
ultrafilter, take a partition $X = Y_1 \amalg Y_2 \amalg Y_3$.  We have
\[
\sum_{i = 1}^3 I(\chi_{Y_i}) 
=
I \biggl( \sum_{i = 1}^3 \chi_{Y_i} \biggr)
=
I(1)
=
1
\]
where the `$1$' in $I(1)$ is the constant function and the last equality
follows from the fact that $I(1) \in \im(1)$.  On the other hand,
$I(\chi_{Y_i}) \in \im(\chi_{Y_i}) \sub \{ 0, 1\}$ for each $i \in \{1, 2,
3\}$, and $0 \neq 1$, $2 \neq 1$, $3 \neq 1$ in $R$, so $I(\chi_{Y_i}) = 1$
for exactly one value of $i \in \{1, 2, 3\}$.  By Corollary~\ref{cor:uf-3},
$\uf{U}$ is an ultrafilter.  Finally, $I = \int_X \idashbk \dee\uf{U}$: for
by linearity, it is enough to check this on characteristic functions, and
this follows from~\eqref{eq:char-fn-int} and the definition of $\uf{U}$.
\done
\end{proof}

\section{Codensity monads as Kan extensions}
\label{sec:kan}

The only ultrafilters on a finite set $B$ are the principal ultrafilters;
hence $U(B) \iso B$.  We prove that $\mnd{U}$ is the universal monad on
$\Set$ with this property.  For the proof, we first need to review some
standard material on codensity, largely covered in early papers such as
\cite{ApTi}, \cite{KockCYR} and~\cite{LintOFS}.  

So far, we have only considered codensity monads for functors whose domain
is essentially small and whose codomain is complete.  We now relax those
hypotheses.  An arbitrary functor $G \from \cat{B} \to \cat{A}$ \demph{has
a codensity monad} if for each $A \in \cat{A}$, the end
\begin{equation}        \label{eq:coden-end}
\int_{B \in \cat{B}} [\cat{A}(A, G(B)),\, G(B)]
\end{equation}
exists.  In that case, we write $T^G(A)$ for this end, so that $T^G$ is a
functor $\cat{A} \to \cat{A}$.  As the end formula reveals, $T^G$ together
with the canonical natural transformation
\begin{equation}        \label{eq:coden-can-transf}
\begin{centerdiag}
\xymatrix@C-2ex@R-2ex{
\cat{B} \ar[rr]^G \ar[rrdd]_G   &       &
\cat{A} \ar[dd]^{T^G} \ar@{}[ld]|-{\swnt\kappa^G}    
\\
&\ &
\\
        &&
\cat{A}
}
\end{centerdiag}
\end{equation}
is the right Kan extension of $G$ along itself.  

It will be convenient to phrase the universal property of the Kan extension
in the following way.  Let $\fmc{G}$ be the category whose objects are
pairs $(S, \sigma)$ of the type
\[
\xymatrix@C-2ex@R-2ex{
\cat{B} \ar[rr]^G \ar[rrdd]_G   &       &
\cat{A} \ar[dd]^S \ar@{}[ld]|-{\swnt\sigma}    
\\
&\ &
\\
        &&
\cat{A}
}
\]
and whose maps $(S', \sigma') \to (S, \sigma)$ are natural transformations
$\theta\from S' \to S$ such that
\[
\begin{array}{c}
\xymatrix@C-1ex@R-1ex{
\cat{B} \ar[rr]^G \ar[rrdd]_G   &&
\cat{A} \ar@{}[ld]|-{\hspace{-1.5em}\swnt\sigma} 
\ddtwocell_S^{S'}{\theta} \\
&\ &
\\
        &&
\cat{A}
}
\end{array}
=
\begin{array}{c}
\xymatrix@C-1ex@R-1ex{
\cat{B} \ar[rr]^G \ar[rrdd]_G   &&
\cat{A} \ar@/^.8pc/[dd]^{S'} \ar@{}[ld]|-{\swnt\sigma'}
\\
&\ &
\\
        &&
\cat{A}.
}
\end{array}
\]
The universal property of $(T^G, \kappa^G)$ is that it is the terminal
object of $\fmc{G}$.

The category $\fmc{G}$ is monoidal under composition.  Being the terminal
object of a monoidal category, $(T^G, \kappa^G)$ has a unique monoid
structure.  This gives $T^G$ the structure of a monad, the \demph{codensity
  monad} of $G$, which we write as $\mnd{T}^G = (T^G, \eta^G, \mu^G)$.
When $\cat{B}$ is essentially small and $\cat{A}$ is complete, this agrees
with the definition in Section~\ref{sec:codensity}.

\begin{example}
Let $\Ring$ be the category of commutative rings, $\Field$ the full
subcategory of fields, and $G\from \Field \incl \Ring$ the inclusion.
Since $\Field$ is not essentially small, it is not instantly clear that $G$
has a codensity monad.  We show now that it does.

Let $A$ be a ring.  Write $A/\Field$ for the comma category in which an
object is a field $k$ together with a homomorphism $A \to k$.
There is a composite forgetful functor
\[
A/\Field \to \Field \incl \Ring,
\]
and the end~\eqref{eq:coden-end}, if it exists, is its limit.  The
connected-components of $A/\Field$ are in natural bijection with the prime
ideals of $A$ (by taking kernels).  Moreover, each component has an initial
object: in the component corresponding to the prime ideal $\ideal{p}$, the
initial object is the composite homomorphism
\[
A \epic A/\ideal{p} \incl \Frac(A/\ideal{p}),
\]
where $\Frac(\dashbk)$ means field of fractions.  Hence the end (or
limit) exists, and it is
\[
T^G(A)
=
\prod_{\ideal{p} \in \Spec(A)} \Frac(A/\ideal{p}).
\]
The unit homomorphism $\eta^G_A \from A \to T^G(A)$ is algebraically
significant: its kernel is the nilradical of $A$, and its image is,
therefore, the free reduced ring on $A$ (\cite{Mats}, Section~1.1).  In
particular, this construction shows that a ring can be embedded into a
product of fields if and only if it has no nonzero nilpotents.  On the
geometric side, $\Spec(T^G(A))$ is the Stone--\v{C}ech compactification of
the discrete space $\Spec(A)$.

For example, 
\[
T^G(\integers) 
= 
\rationals \times 
\prod_{\text{primes } p > 0} \integers/p\integers
\]
(the product of one copy each of the prime fields), and for positive
integers $n$,
\[
T^G(\integers/n\integers) = \integers/\text{rad}(n)\integers
\]
where $\text{rad}(n)$ is the radical of $n$, that is, the product of its
distinct prime factors.
\end{example}

Now consider the case where the functor $G$ is the inclusion of a full
subcategory $\cat{B} \sub \cat{A}$.  Let us say that a monad $\mnd{S} = (S,
\eta^S, \mu^S)$ on $\cat{A}$ \demph{restricts to the identity on $\cat{B}$}
if $\eta^S_B\from B \to S(B)$ is an isomorphism for all $B \in \cat{B}$, or
equivalently if the natural transformation $\eta^S G \from G \to SG$ is an
isomorphism.
When this is so, $(S, (\eta^S G)^{-1})$ is an object of the monoidal
category $\fmc{G}$, and by a straightforward calculation, $((S, (\eta^S
G)^{-1}), \eta^S, \mu^S)$ is a monoid in $\fmc{G}$.  For notational
simplicity, we write this monoid as $(\mnd{S}, (\eta^S G)^{-1})$.

Since $G$ is full and faithful, the natural transformation $\kappa^G$ is an
isomorphism.  But
\[
\begin{array}{c}
\xymatrix@C-1ex@R-1ex{
\cat{B} \ar[rr]^G \ar[rrdd]_G   &&
\cat{A} \ar@{}[ld]|-{\hspace{-1.5em}\swnt\kappa^G} 
\ddtwocell_{T^G}^{1}{\eta^G} \\
&\ &
\\
        &&
\cat{A}
}
\end{array}
=
\begin{array}{c}
\xymatrix@C-1ex@R-1ex{
\cat{B} \ar[rr]^G \ar[rrdd]_G   &&
\cat{A} \ar@/^.8pc/[dd]^{1} \ar@{}[ld]|-{\swnt\id}
\\
&\ &
\\
        &&
\cat{A},
}
\end{array}
\]
so $\eta^G G$ is an isomorphism; that is, $\mnd{T}^G$ restricts to the
identity on $\cat{B}$.  (For example, the set of ultrafilters on a
finite set $B$ is isomorphic to $B$.)  Note that $\kappa^G = (\eta^G
G)^{-1}$.  Also, $(T^G, \kappa^G)$ is the terminal object of $\fmc{G}$, so
$(\mnd{T}^G, \kappa^G)$ is the terminal monoid in $\fmc{G}$.  The following
technical lemma will be useful.

\begin{lemma}   \label{lemma:restricts}
Let $\cat{B}$ be a full subcategory of a category $\cat{A}$, such that the
inclusion functor $G \from \cat{B} \incl \cat{A}$ has a codensity monad.
Let $\mnd{S} = (S, \eta^S, \mu^S)$ be a monad on $\cat{A}$ restricting to
the identity on $\cat{B}$.  For a natural transformation $\alpha \from S
\to T^G$, the following are equivalent:
\begin{enumerate}
\item   \label{part:restricts-monoids}
$\alpha$ is a map $(\mnd{S}, (\eta^S G)^{-1}) \to (\mnd{T}^G, \kappa^G)$ of
monoids in $\fmc{G}$

\item   \label{part:restricts-monads}
$\alpha$ is a map $\mnd{S} \to \mnd{T}^G$ of monads

\item   \label{part:restricts-units}
$\alpha \of \eta^S = \eta^G$.
\end{enumerate}
\end{lemma}

\begin{proof}
The implications (\ref{part:restricts-monoids})$\implies$%
(\ref{part:restricts-monads})$\implies$(\ref{part:restricts-units}) are
trivial.  Assuming~(\ref{part:restricts-units}), the fact that $\kappa^G =
(\eta^G G)^{-1}$ implies that $\alpha$ is a map $(S, (\eta^S G)^{-1}) \to
(T^G, \kappa^G)$ in $\fmc{G}$; and $(T^G, \kappa^G)$ is terminal in
$\fmc{G}$, so $\alpha$ is the unique map of this type.  But also
$(\mnd{T}^G, \kappa^G)$ is the terminal monoid in $\fmc{G}$, so there is a
unique map of monoids $\beta\from (\mnd{S}, (\eta^S G)^{-1}) \to
(\mnd{T}^G, \kappa^G)$.  Then $\alpha = \beta$ by uniqueness of $\alpha$,
giving~(\ref{part:restricts-monoids}).  \done
\end{proof}

Given a monad, it is often possible to find another monad with the same
underlying endofunctor and the same unit, but a different multiplication.
(For example, consider monads $M \times \dashbk$ on $\Set$, where $M$ is a
monoid.)  The same is true of codensity monads in general, since by
Proposition~\ref{propn:adjt-como}, every monad can be constructed as a
codensity monad.  However, codensity monads \emph{of full and faithful
functors} have the special property that their multiplication is
immutable, as follows:

\begin{propn}   \label{propn:mult-immutable}
Let $G\from \cat{B} \to \cat{A}$ be a full and faithful functor that has a
codensity monad.  Let $\mnd{S} = (S, \eta^S, \mu^S)$ be a monad on
$\cat{A}$.  Then:
\begin{enumerate}
\item   \label{part:immutable-iso}
Any natural isomorphism $\alpha\from S \to T^G$ satisfying $\alpha \of
\eta^S = \eta^G$ is an isomorphism of monads.
\item   \label{part:immutable-id}
If $S = T^G$ and $\eta^S = \eta^G$ then $\mu^S = \mu^G$.  
\end{enumerate}
\end{propn}

\begin{proof}
We might as well assume that $G$ is the inclusion of a full subcategory.
Since $\mnd{T}^G$ restricts to the identity on $\cat{B}$, so does
$\mnd{S}$, under the hypotheses of either~(\ref{part:immutable-iso})
or~(\ref{part:immutable-id}).  Lemma~\ref{lemma:restricts} then gives both
parts, taking $\alpha$ to be an isomorphism or the identity, respectively.
\done
\end{proof}

Lemma~\ref{lemma:restricts} also implies:

\begin{propn}   \label{propn:restricts-terminal}
Let $\cat{B}$ be a full subcategory of a category $\cat{A}$, such that the
inclusion functor $G \from \cat{B} \incl \cat{A}$ has a codensity monad.
Then $\mnd{T}^G$ is the terminal monad on $\cat{A}$ restricting to the
identity on $\cat{B}$.
\end{propn}

\begin{proof}
Let $\mnd{S} = (S, \eta^S, \mu^S)$ be a monad on $\cat{A}$ restricting to
the identity on $\cat{B}$.  Then $(\mnd{S}, (\eta^S G)^{-1})$ is a monoid
in $\fmc{G}$, and $(\mnd{T}^G, \kappa^G)$ is the terminal such, so there
exists a unique map $(\mnd{S}, (\eta^S G)^{-1}) \to (\mnd{T}^G, \kappa^G)$
of monoids in $\fmc{G}$.  But by
(\ref{part:restricts-monoids})$\textiff$(\ref{part:restricts-monads}) of
Lemma~\ref{lemma:restricts}, an equivalent statement is that there exists a
unique map $\mnd{S} \to \mnd{T}^G$ of monads.  
\done
\end{proof}

This gives a further characterization of the ultrafilter monad:  

\begin{thm}     \label{thm:uf-terminal-restricting}
The ultrafilter monad is the terminal monad on $\Set$ restricting to the
identity on $\FinSet$.
\done
\end{thm}

To put this result into perspective, note that the \emph{initial} monad on
$\Set$ restricting to the identity on $\FinSet$ is itself the identity, and
that a \emph{finitary} monad on $\Set$ restricting to the identity on
$\FinSet$ can only be the identity.  In this sense, the ultrafilter monad
is as far as possible from being finitary.

\section{Codensity monads as substitutes for adjunction-induced monads}
\label{sec:co-adj}

In the Introduction it was asserted that the codensity monad of a functor
$G$ is a substitute for the monad induced by $G$ and its left adjoint,
valid in situations where no adjoint exists.  The crudest justification is
the following theorem, which goes back to the earliest work on codensity
monads.

\begin{propn}   \label{propn:adjt-como}
Let $G$ be a functor with a left adjoint, $F$.  Then $G$ has a codensity
monad, which is isomorphic to $GF$ with its usual monad structure.
\end{propn}

\begin{proof}
If $G$ is a functor $\cat{B} \to \cat{A}$ then by the Yoneda lemma,
\[
GF(A) 
\iso 
\int_B [\cat{B}(F(A), B),\, G(B)]
\iso
\int_B [\cat{A}(A, G(B)),\, G(B)]
=
T^G(A).
\]
Hence $T^G \iso GF$, and it is straightforward to check that the
isomorphism respects the monad structures. 
\done
\end{proof}

A more subtle justification is provided by the following results,
especially Corollary~\ref{cor:reflection}.  Versions of them appeared in
Section~II.1 of Dubuc~\cite{Dubu}.

We will need some further notation.  Given a category $\cat{A}$, write
$\Mnd(\cat{A})$ for the category of monads on $\cat{A}$ and $\CAT/\cat{A}$
for the (strict) slice of $\CAT$ over $\cat{A}$.  For $\mnd{S} \in
\Mnd(\cat{A})$, write $U^{\mnd{S}}\from \cat{A}^\mnd{S} \to \cat{A}$ for
the forgetful functor on the category of $\mnd{S}$-algebras.  The
assignment $\mnd{S} \mapsto (\cat{A}^{\mnd{S}}, U^{\mnd{S}})$ defines a
functor $\Alg\from \Mnd(\cat{A})^\op \to \CAT/\cat{A}$.

Now let $G\from \cat{B} \to \cat{A}$ be a functor with a codensity monad.
There is a functor $K^G\from \cat{B} \to \cat{A}^{\mnd{T}^G}$, the
\demph{comparison functor} of $G$, defined by
\[
B 
\ \longmapsto\ 
\left(
\slob{T^G G(B)}{G(B)}{\kappa^G_B}
\right)
\]
(where $\kappa^G$ is as in~\eqref{eq:coden-can-transf}).  When $G$ has a
left adjoint $F$, this is the usual comparison functor of the monad $GF$.
In any case, the diagram
\begin{equation}        \label{eq:comparison-lifts}
\begin{centerdiag}
\xymatrix{
\cat{B} \ar[r]^{K^G} \ar[rd]_G  &
\cat{A}^{\mnd{T}^G} \ar[d]^{U^{\mnd{T}^G}}      \\
        &
\cat{A}
}
\end{centerdiag}
\end{equation}
commutes.

\begin{propn}[Dubuc]   \label{propn:cat-mnd-adjn}
Let $\cat{B} \toby{G} \cat{A}$ be a functor that has a codensity monad.
Then 
\[
(\CAT/\cat{A})
\left(
\slob{\cat{B}}{\cat{A}}{G},
\ 
\slob{\cat{A}^\mnd{S}}{\cat{A}}{U^\mnd{S}}
\right)
\iso
\Mnd(\cat{A})\,(\mnd{S}, \mnd{T}^G)
\]
naturally in $\mnd{S} \in \Mnd(\cat{A})$.
\end{propn}

\begin{proof}
Diagram~\eqref{eq:comparison-lifts} states that $K^G$ is a map 
$(\cat{B}, G) \to (\cat{A}^{\mnd{T}^G}, U^{\mnd{T}^G})$ in $\CAT/\cat{A}$.
Let $\mnd{S} \in \Mnd(\cat{A})$ and let $L\from (\cat{B}, G) \to
(\cat{A}^\mnd{S}, U^\mnd{S})$ be a map in $\CAT/\cat{A}$.  We show that
there is a unique map of monads $\overline{L} \from \mnd{S} \to \mnd{T}^G$
satisfying 
\begin{equation}        \label{eq:cat-mnd-factn}
L 
=
\Bigl( 
(\cat{B}, G) \toby{K^G} 
(\cat{A}^{\mnd{T}^G}, U^{\mnd{T}^G}) \toby{\cat{A}^{\overline{L}}} 
(\cat{A}^\mnd{S}, U^\mnd{S})
\Bigr).
\end{equation}

For each $B \in \cat{B}$, we have an $\mnd{S}$-algebra $L(B) = \left(
\slob{SG(B)}{G(B)}{\lambda_B} \right)$.  This defines a natural
transformation
\[ 
\begin{centerdiag}
\xymatrix@C-2ex@R-2ex{
\cat{B} \ar[rr]^G \ar[rrdd]_G   &       &
\cat{A} \ar[dd]^S \ar@{}[ld]|-{\swnt\lambda}    
\\
&\ &
\\
        &&
\cat{A}.
}
\end{centerdiag}
\]
By the universal property of $(T^G, \kappa^G)$, there is a unique map
$\overline{L}\from (S, \lambda) \to (T^G, \kappa^G)$ in $\fmc{G}$.  The algebra
axioms on $L(B)$ imply that $(\mnd{S}, \lambda)$ is a monoid in $\fmc{G}$;
and since $(\mnd{T}^G, \kappa^G)$ is the terminal monoid in $\fmc{G}$, the
map $\overline{L}$ is in fact a map of monads $\mnd{S} \to \mnd{T}^G$.
Equation~\eqref{eq:cat-mnd-factn} states exactly that $\overline{L}$ is a map
$(S, \lambda) \to (T^G, \kappa^G)$ in $\fmc{G}$, so the proof is complete.
\done
\end{proof}

\begin{example} \label{eg:endo}
Every object of a sufficiently complete category has an endomorphism monad.
Indeed, let $\cat{A}$ be a category with small powers, and let $A \in
\cat{A}$.  The functor $A\from \One \to \cat{A}$ has a codensity monad,
given by $X \mapsto [\cat{A}(X, A), A]$.  This is the \demph{endomorphism
  monad} $\END(A)$ of $A$~\cite{KockDDM}.  The name is explained by
Proposition~\ref{propn:cat-mnd-adjn}, which tells us that for any monad
$\mnd{S}$ on $\cat{A}$, the $\mnd{S}$-algebra structures on $A$ correspond
one-to-one with the monad maps $\mnd{S} \to \END(A)$.
\end{example}

Proposition~\ref{propn:cat-mnd-adjn} can be rephrased explicitly as an
adjunction.  Given a category $\cat{A}$, denote by $\cm{\cat{A}}$ the full
subcategory of $\CAT/\cat{A}$ consisting of those functors into $\cat{A}$
that have a codensity monad.  Since every monadic functor has a left
adjoint and therefore a codensity monad, $\Alg$ determines a functor
$\Mnd(\cat{A})^\op \to \cm{\cat{A}}$.  On the other hand, $\mnd{T}^G$
varies contravariantly with $G$, by either direct construction or
Proposition~\ref{propn:cat-mnd-adjn}.  Thus, we have a functor
\[
\mnd{T}^\bl\from \cm{\cat{A}}^\op \to \Mnd(\cat{A}).
\]

\begin{example}
Let $\{2\}$ denote the non-full subcategory of $\Set$ consisting of the
two-element set and its identity map.  Then the inclusion
\[
\left(
\begin{centerdiag}
\xymatrix@R-1ex{\{2\}\vphantom{{}_{p_p}} \ar@{^{(}->}[d]\\ \Set}
\end{centerdiag}
\right)
\quad
\begin{centerdiag}
\xymatrix{\ \ar@{^{(}->}[r] &\ }
\end{centerdiag}
\quad
\left(
\begin{centerdiag}
\xymatrix{\FinSet\vphantom{{}_p} \ar@{^{(}->}[d]\\ \Set}
\end{centerdiag}
\right)
\]
in $\CAT/\Set$ is mapped by $\mnd{T}^\bl$ to the inclusion $\mnd{U} \incl
PP$ of the ultrafilter monad into the double power set monad.  (In the
notation of Example~\ref{eg:endo}, $PP = \END(2)$.)
\end{example}

Proposition~\ref{propn:cat-mnd-adjn} immediately implies that the
construction of codensity monads is adjoint to the construction of
categories of algebras:

\begin{thm}     \label{thm:co-adj}
Let $\cat{A}$ be a category.  Then $\Alg$ and $\mnd{T}^\bl$, as
contravariant functors between $\Mnd(\cat{A})$ and $\cm{\cat{A}}$, are
adjoint on the right.
\done 
\end{thm}

We can usefully express this in another way still.  Recall that the functor
$\Alg\from \Mnd(\cat{A})^\op \to \CAT/\cat{A}$ is full and
faithful~\cite{StreFTM}.  The image is the full subcategory
$\monadic{\cat{A}}$ of $\CAT/\cat{A}$ consisting of the monadic functors
into $\cat{A}$.

\begin{cor}     \label{cor:reflection}
For any category $\cat{A}$, the inclusion
\[
\monadic{\cat{A}} \incl \cm{\cat{A}}
\]
has a left adjoint, given by
\[
G 
\ \longmapsto\ 
\left(
\slob{\cat{A}^{\mnd{T}^G}}{\cat{A}}{U^{\mnd{T}^G}}
\right).
\]
\ \done
\end{cor}

In other words, among all functors into $\cat{A}$ admitting a codensity
monad, the monadic functors form a reflective subcategory.  The reflection
turns a functor $G$ into the monadic functor corresponding to the codensity
monad of $G$.  This is the more subtle sense in which the codensity monad
of a functor $G$ is the best approximation to the monad induced by $G$ and
its (possibly non-existent) left adjoint.

\begin{cor}     \label{cor:CptHff}
In $\CAT/\Set$, the initial map from $(\FinSet \incl \Set)$ to a monadic
functor is
\[
\left(
\begin{centerdiag}
\xymatrix{\FinSet\vphantom{{}_p} 
\ar@{^{(}->}[d] \\ \Set}
\end{centerdiag}
\right)
\quad
\begin{centerdiag}
\xymatrix{\ \ar@{^{(}->}[r] &\ }
\end{centerdiag}
\quad
\left(
\begin{centerdiag}
\xymatrix{\CptHff \ar[d] \\ \Set}
\end{centerdiag}
\right).
\]
\ \done
\end{cor}

As a footnote, we observe that being codense is, in a sense, the opposite
of being monadic.  Indeed, if $G\from \cat{B} \to \cat{A}$ is codense then
$\cat{A}^{\mnd{T}^G} \eqv \cat{A}$, whereas if $G$ is monadic then
$\cat{A}^{\mnd{T}^G} \eqv \cat{B}$.  More precisely:

\begin{propn}
A functor is both codense and monadic if and only if it is an equivalence.
\end{propn}

\begin{proof}
An equivalence is certainly codense and monadic.  Conversely, for any
functor $G\from \cat{B} \to \cat{A}$ with a codensity monad,
diagram~\eqref{eq:comparison-lifts} states that 
\[
G 
=
\Bigl(
\xymatrix@1@C+2ex{
\cat{B} \ar[r]^{K^G} &
\cat{A}^{\mnd{T}^G} \ar[r]^{U^{\mnd{T}^G}} &
\cat{A}
}
\Bigr).
\]
If $G$ is monadic then $G$ has a codensity monad and the comparison functor
$K^G$ is an equivalence; on the other hand, if $G$ is codense then
$\mnd{T}^G$ is isomorphic to the identity, so $U^{\mnd{T}^G}$ is an
equivalence.  The result follows.  
\done
\end{proof}

\section{Double dual vector spaces}
\label{sec:dd}

In this section we prove that the codensity monad of the inclusion
\[
\text{(finite-dimensional vector spaces)} 
\incl
\text{(vector spaces)}
\]
is double dualization.  Much of the proof is analogous to the proof that
the codensity monad of $\FinSet \incl \Set$ is the ultrafilter monad.  (See
the table in the Introduction.)  Nevertheless, aspects of the analogy
remain unclear, and finding a common generalization remains an open
question.

Fix a field $k$ for the rest of this section.  Write $\Vect$ for the
category of $k$-vector spaces, $\FDVect$ for the full subcategory of
finite-dimensional vector spaces, and $\mnd{T} = (T, \eta, \mu)$ for the
codensity monad of $\FDVect \incl \Vect$.  The dualization functor
$\blank^*$ is, as a contravariant functor from $\Vect$ to $\Vect$,
self-adjoint on the right.  This gives the double dualization functor
$\blank^{**}$ the structure of a monad on $\Vect$.  We prove that $\mnd{T}
\iso \blank^{**}$.

Pursuing the analogy, we regard elements $\uf{U}$ of a double dual space
$X^{**}$ as akin to measures on $X$, and we will define an integral
operator $\int_X \idashbk \dee\uf{U}$.  Specifically, let $X \in \Vect$ and
$\uf{U} \in X^{**}$.  We wish to define, for each $B \in \FDVect$, a map
\begin{equation}        \label{eq:vect-integral}
\int_X \idashbk \dee\uf{U}\from \Vect(X, B) \to B.
\end{equation}
In the ultrafilter context, integration has the property that $\int_X
\chi_Y \dee\uf{U} = \mu_{\uf{U}}(Y)$ whenever $\uf{U}$ is an ultrafilter on
a set $X$ and $Y \in P(X)$ (equation~\eqref{eq:char-fn-int}).  Analogously,
we require now that $\int_X \xi \dee\uf{U} = \uf{U}(\xi)$ whenever $\uf{U}
\in X^{**}$ and $\xi \in X^*$; that is, when $B = k$, the integration
operator~\eqref{eq:vect-integral} is $\uf{U}$ itself.  Integration should
also be natural in $B$.  We show that these two requirements determine
$\int_X \idashbk \dee\uf{U}$ uniquely.

\begin{propn}   \label{propn:int-vect-ax}
Let $X$ be a vector space and $\uf{U} \in X^{**}$.  Let $B$ be a
finite-dimensional vector space.  Then there is a unique map of sets 
\[
\int_X \idashbk \dee\uf{U} \from \Vect(X, B) \to B
\]
such that for all $\beta \in B^*$, the square
\[
\xymatrix{
\Vect(X, B) \ar[r]^{\beta\of\dashbk} 
\ar[d]_{\int_X \idashbk \dee\uf{U}}     &
\Vect(X, k) \ar[d]^{\uf{U}}     \\
B \ar[r]_\beta  &
k \\
}
\]
commutes.  When $B = k$, moreover, $\int_X \idashbk \dee\uf{U} = \uf{U}$.
\end{propn}

\begin{proof}
The main statement asserts that $B$ has a certain property; but if 
some vector space isomorphic to $B$ has this property then plainly $B$ does
too.  So it is enough to prove it when $B = k^n$ for some $n \in \nat$.  

Write $\pr_1, \ldots, \pr_n\from k^n \to k$ for the projections, and for $f
\in \Vect(X, k^n)$, write $f_i = \pr_i \of f$.  For any map of sets $\int_X
\idashbk \dee\uf{U}\from \Vect(X, k^n) \to k^n$, 
\begin{align}
\beta \biggl( \int_X \dashbk \dee\uf{U} \biggr) &
=
\uf{U} (\beta \of \dashbk)
\text{ for all } \beta \in (k^n)^*  
\nonumber       \\
\iff
\beta \biggl( \int_X f \dee\uf{U} \biggr)       &
= 
\uf{U}(\beta \of f)
\text{ for all } \beta \in (k^n)^* \text{ and } f \in \Vect(X, k^n)
\nonumber       \\
\iff    
\pr_i \biggl( \int_X f \dee\uf{U} \biggr)       &
= 
\uf{U} ( \pr_i \of f )
\text{ for all } i \in \{1, \ldots, n\} \text{ and } f \in \Vect(X, k^n)
\nonumber       \\
\iff    
\int_X f \dee\uf{U}     &     
= 
(\uf{U}(f_1), \ldots, \uf{U}(f_n))
\text{ for all } f \in \Vect(X, k^n).
\label{eq:int-is-linear}
\end{align}
This proves both existence and uniqueness.  The result on $B = k$ also
follows. 
\done
\end{proof}

Equation~\eqref{eq:int-is-linear} implies that $\int_X \idashbk \dee\uf{U}$
is, in fact, linear with respect to the usual vector space structure on
$\Vect(X, B)$.  (In principle, the notation $\Vect(X, B)$ denotes a mere
set.)  Thus, a linear map 
\[
\uf{U}\from \Vect(X, k) \to k
\]
gives rise canonically to a linear map
\[
\int_X \idashbk \dee\uf{U} \from \Vect(X, B) \to B
\]
for each finite-dimensional vector space $B$.

Integration is natural in two ways, as for sets and ultrafilters
(Lemma~\ref{lemma:uf-int-nat}).  Indeed, writing $\undvect{\cdot}\from
\FDVect \to \Set$ for the underlying set functor, we have the following. 

\begin{lemma}   \label{lemma:int-vect-nat}
\begin{enumerate}
\item   \label{part:int-vect-nat-cod}
Let $X$ be a vector space and $\uf{U} \in X^{**}$.  Then integration
against $\uf{U}$ defines a natural transformation 
\[
\xymatrix@C+5em{
\FDVect 
\rtwocell^{\Vect(X, \dashbk)}_{\undvect{\cdot}}%
{\quad\quad\int\idashbk\dee\uf{U}} &
\Set.
}
\]

\item   \label{part:int-vect-nat-dom}
For any map $X \toby{p} Y$ in $\Vect$ and any $\uf{U} \in X^{**}$, the
triangle 
\[
\xymatrix{
\Vect(X, \dashbk) \ar[rd]_{\int_X \idashbk \dee\uf{U}}  &       &
\Vect(Y, \dashbk) \ar[ll]_{\dashbk \of p} 
\ar[ld]^{\int_Y \idashbk \dee(p^{**}(\uf{U}))}     \\
        &
\undvect{\cdot}
}
\]
in $\ftrcat{\FDVect}{\Set}$ commutes.
\end{enumerate}
\end{lemma}

\begin{proof}
For~(\ref{part:int-vect-nat-cod}), we must prove that for any map $C
\toby{\theta} B$ in $\FDVect$, the square
\[
\xymatrix{
\Vect(X, C) \ar[r]^{\theta\of\dashbk} \ar[d]_{\int \idashbk \dee\uf{U}} &
\Vect(X, B) \ar[d]^{\int \idashbk \dee\uf{U}} \\
C \ar[r]_\theta &
B
}
\]
commutes.  Since the points of $B$ are separated by linear functionals, it
is enough to prove that the square commutes when followed by any linear
$\beta\from B \to k$, and this is a consequence of
Proposition~\ref{propn:int-vect-ax}.

For~(\ref{part:int-vect-nat-dom}), let $B \in \FDVect$.  By the uniqueness
part of Proposition~\ref{propn:int-vect-ax}, it is enough to show that for
all $\beta \in B^*$, the outside of the diagram
\[
\xymatrix{
\Vect(Y, B) \ar[rr]^{\beta\of\dashbk}
\ar[d]_{\dashbk \of p}   &
&
\Vect(Y, k) \ar[ld]_{p^*} \ar[dd]^{p^{**}(\uf{U})}        \\
\Vect(X, B) \ar[r]^{\beta\of\dashbk} 
\ar[d]_{\int \idashbk \dee\uf{U}}     &
\Vect(X, k) \ar[rd]_{\uf{U}}    \\
B \ar[rr]_\beta &
&
k
}
\]
commutes; and the inner diagrams demonstrate that it does.
\done
\end{proof}

Now consider the codensity monad $\mnd{T}$ of $\FDVect \incl
\Vect$.  By definition,
\[
T(X)
=
\int_{B \in \FDVect} [\Vect(X, B), B]
\]
($X \in \Vect$).  Thus, an element $I \in T(X)$ is a family
\[
\Bigl( \Vect(X, B) \toby{I_B} B \Bigr)_{B \in \FDVect}
\]
natural in $B$.  (\latin{A priori}, each $I_B$ is a mere map of sets,
not necessarily linear; but see Lemma~\ref{lemma:linearity-for-free}
below.)  Since the forgetful functor $\Vect \to \Set$ preserves
limits, the underlying set of $T(X)$ is just the set of natural
transformations
\begin{equation}        \label{eq:nat-vect}
\begin{centerdiag}
\xymatrix@C+1em{
\FDVect
\rtwocell^{\Vect(X, \dashbk)}_{\undvect{\cdot}}  &
\Set.
}
\end{centerdiag}
\end{equation}

\begin{propn}   \label{propn:vect-transf}
There is a natural transformation $\blank^{**} \to T$ with components
\begin{equation}        \label{eq:vect-transf}
\begin{mapdefn}
X^{**}  &\to            &T(X)   \\
\uf{U}  &\longmapsto    &\int_X \idashbk \dee\uf{U}     
\end{mapdefn}
\end{equation}
($X \in \Vect$).
\end{propn}

\begin{proof}
Lemma~\ref{lemma:int-vect-nat}(\ref{part:int-vect-nat-cod}) guarantees
that~\eqref{eq:vect-transf} is a well-defined function for each $X$.  The
uniqueness part of Proposition~\ref{propn:int-vect-ax} implies that it is
linear for each $X$.
Lemma~\ref{lemma:int-vect-nat}(\ref{part:int-vect-nat-dom}) tells us that
it is natural in $X$.  
\done
\end{proof}

We are nearly ready to show that the natural
transformation~\eqref{eq:vect-transf} is an isomorphism of monads.  But we
observed after Proposition~\ref{propn:int-vect-ax} that integration against
an ultrafilter is linear, so if this is isomorphism is to hold, the maps
$I_B$ must also be linear.  We prove this now.

\begin{lemma}   \label{lemma:linearity-for-free}
Let $X \in \Vect$ and $I \in T(X)$.  Then for each $B \in \FDVect$, the map
\[
I_B \from \Vect(X, B) \to B
\]
is linear with respect to the usual vector space structure on $\Vect(X, B)$.
\end{lemma}

\begin{proof}
In diagram~\eqref{eq:nat-vect}, both categories have finite products and
both functors preserve them.  Any natural transformation between such
functors is automatically monoidal with respect to the product structures.
From this it follows that whenever $\theta \from B_1 \times \cdots \times
B_n \to B$ is a linear map in $\FDVect$, and whenever $f_i \in \Vect(X,
B_i)$ for $i = 1, \ldots, n$, we have
\[
I_B(\theta \of (f_1, \ldots, f_n)) 
=
\theta(I_{B_1}(f_1), \ldots, I_{B_n}(f_n)).
\]
Let $B \in \FDVect$.  Taking $\theta$ to be first $+\from B \times B \to
B$, then $c \cdot \dashbk \from B \to B$ for each $c \in k$, shows that $I_B$
is linear.  
\done
\end{proof}

\begin{thm}     \label{thm:vect-main}
The codensity monad of $\FDVect \incl \Vect$ is isomorphic to the double
dualization monad $\blank^{**}$ on $\Vect$.  
\end{thm}

\begin{proof}
First we show that the natural transformation $\blank^{**} \to T$ of
Proposition~\ref{propn:vect-transf} is a natural isomorphism, then we show
that it preserves the monad structure.

Let $X$ be a vector space and $I \in T(X)$.  We must show that there is a
unique $\uf{U} \in X^{**}$ such that $I = \int_X \idashbk \dee\uf{U}$.
Uniqueness is immediate from the last part of
Proposition~\ref{propn:int-vect-ax}.  For existence, put 
\[
\uf{U} = I_k \from \Vect(X, k) \to k,
\]
which by Lemma~\ref{lemma:linearity-for-free} is linear (that is, an
element of $X^{**}$).  Naturality of $I$ implies that the square in
Proposition~\ref{propn:int-vect-ax} commutes when $\int_X \idashbk
\dee\uf{U}$ is replaced by $I_B$, so by the uniqueness part of that
proposition, $\int_X \idashbk \dee\uf{U} = I_B$ for all $B \in \FDVect$.

Next, the isomorphism $\blank^{**} \to T$ respects the monad structures.
To prove this, we begin by checking directly that the isomorphism respects
the units of the monads: that is, whenever $X \in \Vect$, the triangle
\[
\xymatrix{
        &
X \ar[ld]_{\text{unit}} \ar[rd]^{\eta_X}        &       
        \\
X^{**}  &
        &
T(X) \ar[ll]|-{\iso}
}
\]
commutes.  Let $x \in X$.  Then $\eta_X(x) \in T(X)$ has $B$-component
\[
\begin{mapdefn}
\Vect(X, B)     &\to            &B      \\
f               &\longmapsto    &f(x)   
\end{mapdefn}
\]
($B \in \FDVect$).  In particular, its $k$-component $\eta_X(x)_k \in
X^{**}$ is evaluation of a functional at $x$, as required.

It now follows from
Proposition~\ref{propn:mult-immutable}(\ref{part:immutable-iso}) that the
natural isomorphism $\blank^{**} \to T$ is an isomorphism of monads.  
\done
\end{proof}

\begin{remark}  \label{rmk:monad-strategies-vs}
The strategy just used to show that the isomorphism is compatible with the
monad structures could also have been used in the case of sets and
ultrafilters (Theorem~\ref{thm:main}).  There we instead used B\"orger's
result that the ultrafilter endofunctor $U$ has a unique monad structure,
which itself was deduced from the fact that $U$ is the terminal endofunctor
on $\Set$ preserving finite coproducts.

Results similar to B\"orger's can also be proved for vector spaces, but
they are complicated by the presence of nontrivial endomorphisms of the
identity functor on $\Vect$ (namely, multiplication by any scalar $\neq
1$).  These give rise to nontrivial endomorphisms of every nonzero
endofunctor of $\Vect$.  Hence double dualization cannot be the terminal
$\oplus$-preserving endofunctor.  However, it \emph{is} the terminal
$\oplus$-preserving endofunctor $S$ equipped with a natural transformation
$1 \to S$ whose $k$-component is an isomorphism.  The proof is omitted.
\end{remark}

We have already seen that the notion of compact Hausdorff space arises
canonically from the notion of finiteness of a set: compact Hausdorff
spaces are the algebras for the codensity monad of $\FinSet \incl \Set$.
What is the linear analogue?

\begin{defn}
A \demph{linearly compact vector space} over $k$ is a $k$-vector space in
$\Top$ with the following properties:
\begin{enumerate}
\item the topology is \demph{linear}: the open affine subspaces form a
  basis for the topology 
\item every family of closed affine subspaces with the finite intersection
  property has nonempty intersection
\item the topology is Hausdorff.
\end{enumerate}
We write $\LCVect$ for the category of linearly compact vector spaces and
continuous linear maps. 
\end{defn}

For example, a finite-dimensional vector space can be given the structure
of a linearly compact vector space in exactly one way: by equipping it with
the discrete topology.

Linearly compact vector spaces were introduced by Lefschetz (Chapter~II,
Definition~27.1 of~\cite{Lefs}).  A good modern reference is the book of
Bergman and Hausknecht~\cite{BeHa}.

\begin{thm}     \label{thm:lcvect}
The category of algebras for the codensity monad of $\FDVect \incl \Vect$
is equivalent to $\LCVect$, the category of linearly compact vector spaces.
\end{thm}

\begin{proof}
The codensity monad is the double dualization monad, which by definition is
the monad obtained from the dualization functor $\blank^*\from \Vect^\op
\to \Vect$ and its left adjoint.  The dualization functor is, in fact,
monadic.  A proof can be extracted from Linton's proof that the dualization
functor on Banach spaces is monadic~\cite{LintAFS3}.  Alternatively, we can
use the following direct argument, adapted from a proof by
Trimble~\cite{TrimMO}.

We apply the monadicity theorem of Beck.  First, $\Vect^\op$ has all
coequalizers.  Second, the dualization functor preserves them:
for the object $k$ of the abelian category $\Vect$ is injective, so by
Lemma~2.3.4 of~\cite{Weib}, the dualization functor is exact.  Third,
dualization reflects isomorphisms.  Indeed, let $f\from X \to Y$ be a
linear map such that $f^*\from Y^* \to X^*$ is an isomorphism.  Dualizing
the exact sequence
\[
0 \to \ker f \to X \toby{f} Y \to \coker f \to 0
\]
yields another exact sequence, in which the middle map is an isomorphism.
Hence $(\ker f)^* \iso 0 \iso (\coker f)^*$.  From this it follows that $\ker
f \iso 0 \iso \coker f$, so $f$ is an isomorphism, as required. 

On the other hand, it was shown by Lefschetz that $\Vect^\op \eqv \LCVect$
(Chapter~II, number~29 of \cite{Lefs}; or see Proposition~24.8
of~\cite{BeHa}).  This proves the theorem.  
\done
\end{proof}

A slightly more precise statement can be made.  Lefschetz's equivalence
$\Vect^\op \to \LCVect$ sends a vector space $X$ to its dual $X^*$,
suitably topologized.  Hence, under the equivalence $\Vect^\mnd{T} \eqv
\LCVect$, the forgetful functor $U^\mnd{T}\from \Vect^\mnd{T} \to \Vect$
corresponds to the obvious forgetful functor $\LCVect \to \Vect$.

In summary,
\begin{center}
\it
sets are to compact Hausdorff spaces\\
as\\
vector spaces are to linearly compact vector spaces.  
\end{center}
It seems not to be known whether this is part of a larger pattern.  Is it
the case, for example, that for all algebraic theories, the codensity monad
of the inclusion
\[
\text{(finitely presentable algebras)}
\incl
\text{(algebras)}
\]
is equivalent to a suitably-defined category of `algebraically compact'
topological algebras?

\section{Ultraproducts}
\label{sec:ups}

The ultraproduct construction, especially important in model theory, can
also be seen as a codensity monad.

Let $X$ be a set, $S = (S_x)_{x \in X}$ a family of sets, and $\uf{U}$ an
ultrafilter on $X$.  The \demph{ultraproduct} $\ulp{\uf{U}} S$ is the
colimit of the functor $(\uf{U}, \sub)^\op \to \Set$ defined on objects by
$H \mapsto \prod_{x \in H} S_x$ and on maps by projection.
(See~\cite{FaHa} or Section~1.2 of~\cite{Eklo}).  Explicitly,
\[
\ulp{\uf{U}} S
=
\biggl( \sum_{H \in \uf{U}} \prod_{x \in H} S_x \biggr) \biggm/ \sim
\]
where $\sum$ means coproduct and
\[
(s_x)_{x \in H} \sim (t_x)_{x \in K} 
\iff
\{ x \in H \cap K \such s_x = t_x \} \in \uf{U}.
\]
For a trivial example, if $\uf{U}$ is the principal ultrafilter on $x$ then
$\ulp{\uf{U}} S = S_x$.

Logic texts often assume that all the sets $S_x$ are nonempty \cite{ChKe,
HodgMT}, in which case the ultraproduct can be described more
simply as $(\prod_{x \in X} S_x)/\!\!\sim$.  The appendix of
Barr~\cite{BarrMS} explains why the present definition is the right one in the
general case.

Ultraproducts can also be understood sheaf-theoretically (as in~2.6.2
of~\cite{Schou}).  A family $(S_x)_{x \in X}$ of sets amounts to a sheaf
$S$ on the discrete space $X$, with stalks $S_x$.  The unit map $\eta_X
\from X \to U(X)$ embeds the discrete space $X$ into its Stone--\v{C}ech
compactification, and pushing forward gives a sheaf $(\eta_X)_* S$ on
$U(X)$.  The stalk of this sheaf over $\uf{U}$ is exactly the ultraproduct
$\ulp{\uf{U}} S$. 

Since the category $(\uf{U}, \sub)^\op$ is filtered, the definition of
ultraproduct can be generalized from sets to the objects of any other
category $\cat{E}$ with small products and filtered colimits.  Thus, a
family $S = (S_x)_{x \in X}$ of objects of $\cat{E}$, indexed over a set
$X$, gives rise to a new family $(\ulp{\uf{U}} S)_{\uf{U} \in U(X)}$ of
objects of $\cat{E}$.

For the rest of this section, fix a category $\cat{E}$ with small products
and filtered colimits.

Let $\Fam(\cat{E})$ be the category in which an object is a set $X$
together with a family $(S_x)_{x \in X}$ of objects of $\cat{E}$, and a map
$(S_x)_{x \in X} \to (R_y)_{y \in Y}$ is a map of sets $f\from X \to Y$
together with a map $\phi_x\from R_{f(x)} \to S_x$ for each $x \in X$.
(Note the direction of the last map; this marks a difference from other
authors' use of the $\Fam$ notation.)  Let $\FinFam(\cat{E})$ be the full
subcategory consisting of those families $(S_x)_{x \in X}$ for which the
indexing set $X$ is finite.

The main theorem states, essentially, that the codensity monad of the
inclusion $\FinFam(\cat{E}) \incl \Fam(\cat{E})$ is given on objects by 
\begin{equation}        \label{eq:up-objs}
(S_x)_{x \in X} \longmapsto \Bigl( \ulp{\uf{U}} S \Bigr)_{\uf{U} \in U(X)}.  
\end{equation}
So the ultraproduct construction arises naturally from the notion of
finiteness of a family of objects.

In particular, the ultraproduct construction determines a monad.
This monad, the \demph{ultraproduct monad} $\mnd{V}$ on $\Fam(\cat{E})$,
was first described by Ellerman~\cite{Elle} and Kennison~\cite{Kenn}.  We
review their definition, then prove that the codensity monad of
$\FinFam(\cat{E}) \incl \Fam(\cat{E})$ exists and is isomorphic to
$\mnd{V}$.

Our first task is to define the underlying functor $V\from \Fam(\cat{E})
\to \Fam(\cat{E})$.  On objects, $V$ is given by~\eqref{eq:up-objs}.  Now
take a map $(f, \phi)\from (S_x)_{x \in X} \to (R_y)_{y \in Y}$ in
$\Fam(\cat{E})$, which by definition consists of maps
\[
X \toby{f} Y,
\qquad
R_{f(x)} \toby{\phi_x} S_x 
\quad
(x \in X).
\]
Its image under $V$ consists of maps
\[
U(X) \toby{f_*} U(Y),
\qquad
\ulp{f_*\uf{U}} R \to \ulp{\uf{U}} S 
\quad
(\uf{U} \in U(X)).
\]
The first of these maps, $f_*$, is $U(f)$.  The second is the map
\begin{equation}        \label{eq:up-func-map}
\colt{K \in f_*\uf{U}} \prod_{y \in K} R_y
\to
\colt{H \in \uf{U}} \prod_{x \in H} S_x
\end{equation}
whose $K$-component is the composite
\[
\xymatrix@C+5em{
\displaystyle
\prod\limits_{y \in K} R_y     
\ar[r]^-{\prod\limits_{y \in K} (\phi_x)_{x \in f^{-1}(y)}}       &
\displaystyle
\prod\limits_{y \in K} \prod\limits_{x \in f^{-1}(y)} S_x
\iso
\prod\limits_{x \in f^{-1}K} S_x       
\ar[r]^-{\copr_{f^{-1}K}} &
\displaystyle
\ulp{\uf{U}}S
}
\]
where $\copr$ denotes a coprojection.  In the case $\cat{E} = \Set$, the
map~\eqref{eq:up-func-map} sends the equivalence class of a family
$(r_y)_{y \in K}$ to the equivalence class of the family $\bigl(
\phi_x(r_{f(x)}) \bigr)_{x \in f^{-1}K}$.

Next we describe the unit of the ultraproduct monad.  Its component at an
object $(S_x)_{x \in X}$ consists of maps
\[
X \toby{\eta_X} U(X),
\qquad
\ulp{\eta_X(x)} S \to S_x
\quad
(x \in X).
\]
The first map is the unit of the ultrafilter monad $\mnd{U}$, in which
$\eta_X(x)$ is the principal ultrafilter on $x$.  The second map is the
canonical isomorphism.  

Proposition~\ref{propn:mult-immutable} will save us from needing to know the
multiplication of the ultraproduct monad.

To prove the main theorem, our first step is to recast the definition of
ultraproduct.  The usual definition treats an ultrafilter as a collection
of subsets; but to connect with the codensity characterization of
ultrafilters, we need a definition of ultraproduct that treats ultrafilters
as integration operators.

\begin{lemma}   \label{lemma:ulp-int}
Let $\cat{B}$ be a full subcategory of $\FinSet$ containing at least one
set with at least three elements.  Let $(S_x)_{x \in X}$ be an object of
$\Fam(\cat{E})$, and let $\uf{U} \in U(X)$.  Then there is a canonical
isomorphism
\begin{equation}        \label{eq:two-ulps}
\ulp{\uf{U}} S
\iso
\colt{\substack{B \in \cat{B},\\ f \from X \to B}}
\prod_{\ x \in f^{-1}(\int f \dee\uf{U})} S_x
\end{equation}
where the right-hand side is a colimit over the category of elements of
$\Set(X, \dashbk)\from \cat{B} \to \Set$.  
\end{lemma}

\begin{proof}
Write $\aulp{\uf{U}} S$ for the right-hand side of~\eqref{eq:two-ulps}.
Thus, whenever $X \toby{f} B \toby{g} B'$ with $B, B' \in \cat{B}$,
we have a commutative triangle
\begin{equation}        \label{eq:up-cocone}
\begin{centerdiag}
\xymatrix@R-6ex@C+2em{
\displaystyle\prod_{x \in (gf)^{-1}(\int gf \dee\uf{U})} S_x    
\ar[rd]^-{\copr_{gf}}
\ar[dd]_{\pr}   &
        \\
        &
\displaystyle\aulp{\uf{U}} S    \\
\displaystyle\prod_{x \in f^{-1}(\int f \dee\uf{U})} S_x
\ar[ru]_-{\copr_f}
}
\end{centerdiag}
\end{equation}
where the vertical map is a product projection and the other maps are
colimit coprojections.

Define $\theta\from \ulp{\uf{U}} S \to \aulp{\uf{U}} S$ as follows.
For each $H \in \uf{U}$, choose some $B \in \cat{B}$ and $f\from X \to B$
such that $f^{-1}(\int f \dee\uf{U}) = H$; then the $H$-component of
$\theta$ is
\[
\theta_H = \copr_f \from
\prod_{x \in H} S_x 
\to
\aulp{\uf{U}} S.
\]
We have to check (i)~that $\theta_H$ is well-defined for each individual
$H$, and (ii)~that $\theta_H$ is natural in $H$, thus defining a map
$\theta\from \ulp{\uf{U}} S \to \aulp{\uf{U}} S$.

For~(i), let $H \in \uf{U}$.  Choose $\Omega \in \cat{B}$ with at least two
elements, say $0$ and $1$, and define characteristic functions by the usual
formula~\eqref{eq:char-fn-defn}.  There is at least one pair $(B, f)$ such
that $f^{-1}(\int f \dee\uf{U}) = H$: for example, $(\Omega, \chi_H)$.  Let
$(B, f)$ be another such pair.  In the triangle~\eqref{eq:up-cocone} with
$g = \chi_{\{\int f \dee\uf{U}\}}$, we have $gf = \chi_H$, so the
vertical map is an identity.  Hence $\copr_{\chi_H} = \copr_{f}$, as
required.

For~(ii), let $H, H' \in \uf{U}$ with $H \sub H'$; we must prove the
commutativity of
\begin{equation}        \label{eq:phi-nat}
\begin{centerdiag}
\xymatrix@R-6ex@C+2em{
\displaystyle\prod_{x \in H'} S_x    
\ar[rd]^-{\theta_{H'}}
\ar[dd]_{\pr}   &
        \\
        &
\displaystyle\aulp{\uf{U}} S.   \\
\displaystyle\prod_{x \in H} S_x
\ar[ru]_-{\theta_H}
}
\end{centerdiag}
\end{equation}
Choose $B \in \cat{B}$ with at least three elements, say $a$,
$b$ and $c$, and define $f\from X \to B$ by
\[
f(x)
=
\begin{cases}
a       &\text{if } x \in H        \\
b       &\text{if } x \in H'\without H  \\
c       &\text{if } x \not\in H'.
\end{cases}
\]
Then $\chi_{\{a, b\}} \of f = \chi_{H'}$, and the commutative
triangle~\eqref{eq:up-cocone} with $g = \chi_{\{a, b\}}$ is 
exactly~\eqref{eq:phi-nat}.  

Let $\tilde{\theta}\from \aulp{\uf{U}} S \to \ulp{\uf{U}} S$ be the unique map
such that, whenever $B \in \cat{B}$ and $f\from X \to B$, the $(B,
f)$-component of $\tilde{\theta}$ is the coprojection
\[
\prod_{x \in f^{-1}(\int f \dee\uf{U})} S_x 
\to
\colt{H \in \uf{U}} \prod_{x \in H} S_x
=
\ulp{\uf{U}} S.
\]
It is straightforward to check that $\tilde{\theta}$ is a two-sided inverse
of $\theta$. 
\done
\end{proof}

We now turn to the category $\Fam(\cat{E})$.  Recall that any functor
$\Sigma\from \cat{A}^\op \to \CAT$ has a category of elements (or
Grothendieck construction) $\elt{\Sigma}$, whose objects are pairs $(A, S)$
with $A \in \cat{A}$ and $S \in \Sigma(A)$, and whose maps $(A, S) \to (B,
R)$ are pairs $(f, \phi)$ with $f\from A \to B$ and $\phi\from S \to f^*
R$.  (We write $f^*$ for $\Sigma(f)$.)  It comes with a projection functor
$\pr\from \elt{\Sigma} \to \cat{A}$.  For example, there is a functor
$\Sigma\from \Set^\op \to \CAT$ given on objects by
\begin{equation}        \label{eq:slice}
\Sigma(X) = \bigl(\cat{E}^X\bigr)^\op
\end{equation}
and on maps $f\from X \to Y$ by taking $\Sigma(f)$ to be the dual of the
reindexing functor $f^*\from \cat{E}^Y \to \cat{E}^X$.  Its category of
elements is $\Fam(\cat{E})$.

To compute the codensity monad of $\FinFam(\cat{E}) \incl \Fam(\cat{E})$,
we will need to know about limits in $\Fam(\cat{E})$.  (Compare Section~2
of~\cite{Kenn}.)  We work with categories of elements more generally, using
the following standard lemma.

\begin{lemma}   \label{lemma:elts-lim}
Let $\scat{I}$ be a category.  Let $\cat{A}$ be a category with limits over
$\scat{I}$.  Let $\Sigma\from \cat{A}^\op \to \CAT$ be a functor such that
$\Sigma(A)$ has limits over $\scat{I}$ for each object $A$ of $\cat{A}$,
and $f^*$ preserves limits over $\scat{I}$ for each map $f$ in $\cat{A}$.
Then $\elt{\Sigma}$ has limits over $\scat{I}$, and the projection
$\elt{\Sigma} \to \cat{A}$ preserves them.
\end{lemma}

\begin{proof}
Let $(A_i, S_i)_{i \in \scat{I}}$ be a diagram over $\scat{I}$ in
$\elt{\Sigma}$; thus, $A_i \in \cat{A}$ and $S_i \in \Sigma(A_i)$ for each
$i \in \scat{I}$.  Take a limit cone $(A \toby{p_i} A_i)_{i \in \scat{I}}$
in $\cat{A}$.  We obtain a diagram $\bigl(p_i^*(S_i)\bigr)_{i \in
  \scat{I}}$ in $\Sigma(A)$, and its limit in $\Sigma(A)$ is a limit of the
original diagram in $\elt{\Sigma}$.  \done
\end{proof}

Taking the category of elements is a functorial process: given $\Sigma\from
\cat{A}^\op \to \CAT$ and $G\from \cat{B} \to \cat{A}$, we obtain a
commutative square
\[
\xymatrix{
\elt{\Sigma\of G}
\ar[r]^-{G'}
\ar[d]_{\pr}    &
\elt{\Sigma}
\ar[d]^{\pr}    \\
\cat{B}
\ar[r]_{G}      &
\cat{A}
}
\]
where $G'(B, R) = (GB, R)$ whenever $B \in \cat{B}$ and $R \in \Sigma(GB)$.
For example, if $\Sigma$ is the functor~\eqref{eq:slice} and $G$ is the
inclusion $\FinSet \incl \Set$ then $G'$ is the inclusion $\FinFam(\cat{E})
\incl \Fam(\cat{E})$.  The next two results will enable us to compute the
codensity monad of $G'$.

\begin{propn}   \label{propn:elts-coden}
Let $G$ be a functor from an essentially small category $\cat{B}$ to a
complete category $\cat{A}$.  
Let $\Sigma\from \cat{A}^\op \to \CAT$ be a functor such that for each
object $A$ of $\cat{A}$, the category $\Sigma(A)$ is complete, and for each
map $f\from A \to A'$ in $\cat{A}$, the functor $f^* = \Sigma(f) \from
\Sigma(A') \to \Sigma(A)$ has a left adjoint $f_!$.  

Then $G'\from \elt{\Sigma\of G} \to \elt{\Sigma}$ has a codensity monad,
given at $(A, S) \in \elt{\Sigma}$ by
\[
T^{G'}(A, S) 
=
\biggl(
T^G(A),
\lt{\substack{B \in \cat{B},\\ f\from A \to GB}} \pr_f^* f_! (S)
\biggr)
\]
where the limit is over the category of elements of $\cat{A}(A,
G\dashbk)\from \cat{B} \to \Set$, and $\pr_f\from T^G(A) \to G(B)$ is
projection. 
\end{propn}

\begin{proof}
$G'$ has a codensity monad if for each $(A, S) \in \elt{\Sigma}$, the limit
\[
T^{G'}(A, S) 
=
\lt{\substack{(B, R) \in \elt{\Sigma\of G},\\(f, \phi)\from (A, S) \to (GB, R)}}
(GB, R)
\]
in $\elt{\Sigma}$ exists.  This is a limit over the category of elements of 
\[
\elt{\Sigma}\bigl((A, S), G'\dashbk\bigr)\from
\elt{\Sigma\of G} \to \Set.
\]
An object of this category of elements consists of an object $B$ of
$\cat{B}$, an object $R$ of $\Sigma(GB)$, a map $f\from A \to GB$ in
$\cat{A}$, and a map $\phi\from S \to f^* R$ in $\Sigma(A)$ (or
equivalently, a map $\bar{\phi}\from f_! S \to R$ in $\Sigma(GB)$).  It
follows that
\[
T^{G'}(A, S)
=
\lt{\substack{B \in \cat{B},\\ f \from A \to GB\ }} 
\lt{\substack{R \in \Sigma(GB),\\ \bar{\phi} \from f_! S \to R}}
(GB, R)
\]
provided that the right-hand side exists.  Here the outer limit is over
the category of elements of $\cat{A}(A, G\dashbk)\from \cat{B} \to \Set$,
and the inner limit is over the coslice category $f_! S/\Sigma(GB)$.  But
the coslice category has an initial object, the identity on $f_! S$, so
\[
T^{G'}(A, S)
=
\lt{\substack{B \in \cat{B},\\ f \from A \to GB}} (GB, f_! S)
\]
provided that this limit in $\elt{\Sigma}$ exists.  By the completeness
hypotheses and Lemma~\ref{lemma:elts-lim}, it does exist, and by the
construction of limits given in the proof of that lemma,
\[
T^{G'}(A, S)
=
\Biggl(
T^G(A), \lt{\substack{B \in \cat{B},\\ f \from A \to GB}} \pr_f^* f_! S
\Biggr),
\]
as required.
\done
\end{proof}

The proposition describes only the underlying functor of the codensity
monad $\mnd{T}^{G'}$.  The component of the unit at an object $(A, S)$ of
$\elt{\Sigma}$ consists first of a map $A \to T^G(A)$, which is just the
unit map $\eta^G_A$ of the codensity monad of $G$, and then of a map
\[
i_S \from
S \to 
\bigl(\eta^G_A\bigr)^* 
\left(
\lt{\substack{B \in \cat{B},\\ f\from A \to GB}} \pr_f^* f_! S
\right).
\]
Since $\bigl(\eta^G_A\bigr)^*$ has a left adjoint, the codomain of $i_S$ is
\[
\lt{B, f} \bigl(\eta^G_A\bigr)^* \pr_f^* f_! S
=
\lt{B, f} \bigl(\pr_f \of \eta^G_A\bigr)^* f_! S
=
\lt{B, f} f^* f_! S.
\]
Under these isomorphisms, 
the $(B, f)$-component of $i_S$ is the unit map $S \to f^* f_! S$.

We will need a variant of Proposition~\ref{propn:elts-coden}.

\begin{propn}   \label{propn:elts-coden-filt}
Proposition~\ref{propn:elts-coden} holds under the following alternative
hypotheses: $\cat{B}$ is now required to have finite limits and $G$ to
preserve them, but for each $A \in \cat{A}$, the category $\Sigma(A)$ is
only required to have \emph{cofiltered} limits.
\end{propn}

\begin{proof}
Since $\cat{B}$ has finite limits and $G$ preserves them, the category of
elements of $\cat{A}(A, G\dashbk)\from \cat{B} \to \Set$ is cofiltered for
each $A \in \cat{A}$.   The proof is now identical to that of
Proposition~\ref{propn:elts-coden}.
\done
\end{proof}

The case $\cat{E} = \Set$ of the following theorem is due to the referee
(who gave a different proof).

\begin{thm}     \label{thm:ulp}
The inclusion $\FinFam(\cat{E}) \incl \Fam(\cat{E})$ has a codensity monad,
isomorphic to the ultraproduct monad on $\Fam(\cat{E})$.
\end{thm}

\begin{proof}
We apply Proposition~\ref{propn:elts-coden-filt} when $G$ is the inclusion
$\FinSet \incl \Set$ and $\Sigma$ is the functor of~\eqref{eq:slice}.
First we verify the hypotheses of that proposition, using our standing
assumption that $\cat{E}$ is a category with small products and filtered
colimits.  For each map of sets $f\from X \to X'$, the functor $f^*\from
\cat{E}^{X'} \to \cat{E}^X$ has a right adjoint $f_*$, given at $S \in
\cat{E}^X$ by
\[
f_*(S) = 
\Biggl(
\prod_{x \in f^{-1}(x')} S_x
\Biggr)_{x' \in X'}.
\]
Hence $f^* \from \Sigma(X') \to \Sigma(X)$ has a \emph{left} adjoint.  The
other hypotheses are immediate.  

By Proposition~\ref{propn:elts-coden-filt}, the inclusion $G'\from
\FinFam(\cat{E}) \incl \Fam(\cat{E})$ has a codensity monad, given at
$S \in \cat{E}^X$ by
\[
T^{G'}(S)
=
\colt{\substack{B \in \FinSet,\\ f\from X \to B}} \pr_f^* f_*(S)
\]
where the colimit is taken in $\cat{E}^{T^G(X)} = \Sigma\bigl(T^G
X\bigr)^\op$.  By Theorem~\ref{thm:main}, $T^G$ is the ultrafilter monad.
For $B \in \FinSet$ and $f\from X \to B$, the projection
\[
\pr_f\from T^G(X) = U(X) \to B
\]
is $\uf{U} \mapsto \int_X f \dee\uf{U}$; hence
\[
\pr_f^* f_* S = 
\Biggl( 
\prod_{x \in f^{-1}(\int f \dee\uf{U})} S_x 
\Biggr)_{\uf{U} \in U(X)}.
\]
So by Lemma~\ref{lemma:ulp-int}, $T^{G'}(S)$ is canonically isomorphic to
$\bigl(\ulp{\uf{U}} S\bigr)_{\uf{U} \in U(X)} = V(S)$.  

This shows that the codensity monad $\mnd{T}^{G'}$ of
$\FinFam(\cat{E}) \incl \Fam(\cat{E})$ exists and has the same
underlying functor as the ultraproduct monad $\mnd{V}$.  Using the
description of the unit of $\mnd{T}^{G'}$ given after
Proposition~\ref{propn:elts-coden}, one can check that their units also
agree.  It follows from Proposition~\ref{propn:mult-immutable} that the
monads $\mnd{T}^{G'}$ and $\mnd{V}$ are isomorphic.
\done
\end{proof}

\begin{examples}
\begin{enumerate}
\item
The codensity monad of $\FinFam(\Set) \incl \Fam(\Set)$ is the
ultraproduct monad on $\Fam(\Set)$.  

\item
The same is true when $\Set$ is replaced by $\Ring$, $\Grp$, or the
category $\cat{E}$ of algebras for any other finitary algebraic theory.  In
such categories, small products and filtered colimits are computed as in
$\Set$, so ultraproducts are computed as in $\Set$ too.

\item
Take the category $\cat{E}$ of structures and homomorphisms for a
(finitary) signature, in the sense of model theory.  This has products and
filtered colimits, both computed as in $\Set$, and the codensity monad of
$\FinFam(\cat{E}) \incl \Fam(\cat{E})$ is the ultraproduct construction for
such structures.  It remains to be seen whether {\L}o\'s's theorem (the
fundamental theorem on ultraproducts) can usefully be understood in this
way.
\end{enumerate}
\end{examples}

There is an alternative version of Theorem~\ref{thm:ulp}.  Let $\cat{B}$ be
a full subcategory of $\FinSet$ containing at least one set with at least
three elements, and write $\Fam_\cat{B}(\cat{E})$ for the full subcategory
of $\Fam(\cat{E})$ consisting of the families $(S_x)_{x \in X}$ with $X \in
\cat{B}$.  Assume that $\cat{E}$ has all small colimits (not just
filtered colimits).  Then the codensity monad of the inclusion
$\Fam_\cat{B}(\cat{E}) \incl \Fam(\cat{E})$ is the ultraproduct monad.   The
proof is the same as that of Theorem~\ref{thm:ulp}, but replacing $\FinSet$
by $\cat{B}$ and Proposition~\ref{propn:elts-coden-filt} by
Proposition~\ref{propn:elts-coden}.

We finish by describing the algebras for the ultraproduct monad,
restricting our attention to $\cat{E} = \Set$.

Let $\Sheaf$ be the category in which an object is a topological space $X$
equipped with a sheaf $S$ of sets, and a map $(X, S) \to (Y, R)$ is a
continuous map $f\from X \to Y$ together with a map $f^* R \to S$ of
sheaves on $X$.  It has a full subcategory consisting of the objects $(X,
S)$ where $X$ is discrete; this is nothing but $\Fam(\Set)$.  It also
has a full subcategory $\CHSheaf$ consisting of the objects $(X, S)$ for
which $X$ is compact and Hausdorff.

The following corollary was also pointed out by the referee.

\begin{cor}     \label{cor:chsheaf}
The category of algebras for the codensity monad of $\FinFam(\Set) \incl
\Fam(\Set)$ is equivalent to $\CHSheaf$, the category of sheaves on compact
Hausdorff spaces.
\end{cor}

\begin{proof}
Theorem~1.4 of Kennison~\cite{Kenn} states that $\CHSheaf$ is the category
of algebras for the ultraproduct monad.  The result follows from
Theorem~\ref{thm:ulp}.
\done
\end{proof}

So the notion of finiteness of a family of sets leads inevitably not
only to the notion of ultraproduct, but also to the notion of sheaf on a
compact Hausdorff space.

\paragraph*{Acknowledgements} I thank Ken Brown, Eugenia Cheng, Alastair
Craw, Michel H\'ebert, Thomas Holder, Peter Johnstone, Anders Kock, Uli
Kr\"ahmer, Kobi Kremnitzer, Steve Lack, Fred Linton, Mark Meckes, Tim
Porter, Emily Riehl, Ben Steinberg, Ross Street, Myles Tierney and Boris
Zilber.  I am especially grateful to Todd Trimble for encouragement, for
pointing out a result of Lawvere (Corollary~\ref{cor:lawvere} above), and
for showing me a proof~\cite{TrimMO} that the algebras for the double
dualization monad are the linearly compact vector spaces
(Theorem~\ref{thm:lcvect}).  Finally, I am enormously grateful to the
anonymous referee, who very generously gave me the main theorem of
Section~\ref{sec:ups}.

\ucontents{section}{References}%

\end{document}